\documentclass [12pt]{article}
 \usepackage{amssymb}
 \usepackage{amsmath,amsthm}
 \usepackage{mathrsfs}
 \usepackage{color}
 \usepackage{graphicx}

\normalsize

\newcommand\codeg{\operatorname{codeg}}

\newtheorem{theorem}{Theorem}[section]
\newtheorem{lemma}[theorem]{Lemma}

\newtheorem{corollary}[theorem]{Corollary}

\newtheorem{remark}[theorem]{Remark}

\numberwithin{equation}{section}

\begin{document}
\title
{On partial Steiner $(n,r,\ell)$-system process\thanks{Section 4
has been submitted to Discrete Mathematics with title ``Connectivity of the linear uniform hypergraph process".}}

\author
{{Fang Tian}\\
{\small Department of Applied Mathematics}\\
{\small Shanghai University of Finance and Economics, Shanghai, 200433, China}\\
{\small\tt tianf@mail.shufe.edu.cn}
}

\date{}
 \maketitle
\begin{abstract}  For given integers $r$ and $\ell$ such that
  $2\leqslant\ell\leqslant r-1$, an $r$-uniform hypergraph $H$ is called
  a partial Steiner $(n,r,\ell)$-system,
if every subset of size $\ell$  lies in at most one
edge of $H$.  In particular, partial Steiner $(n,r,2)$-systems
 are also called linear hypergraphs.
 The partial Steiner $(n,r,\ell)$-system process starts with  an empty hypergraph on vertex set $[n]$
 at time $0$, the $ \binom{n}{r}$ edges arrive one by one according to a uniformly
 chosen permutation, and each edge is added if and only if it does not overlap any of the previously-added
 edges in $\ell$ or more vertices. In this paper, we show with high probability, independent of $\ell$,
 the sharp threshold of connectivity in the algorithm is $ \frac{n}{r}\log n$ and the very edge which links the last isolated
  vertex with another vertex makes the partial Steiner $(n,r,\ell)$-system connected.
\end{abstract}

\vskip0.4cm \hskip 10pt {\bf Keywords:}\quad asymptotic enumeration, graph process,
 linear hypergraph, connectivity.

\hskip 10pt  hitting time, connectivity.

\hskip 10pt {\bf Mathematics Subject Classifications:}\ 05A16, 05D40

\date{}
 \maketitle

\section{Introduction}\label{s:1}

Hypergraphs, which are also known as set systems and block designs,
are fundamental to the study of complex discrete systems.
Let $r$ and $\ell$ be given integers such that
 $2\leqslant\ell\leqslant r-1$. A hypergraph ${H}$ on vertex set $[n]$ is an \textit{$r$-uniform hypergraph}
(\textit{$r$-graph} for short) if each edge is a set of  $r$ vertices.
An $r$-graph is called a Steiner $(n,r,\ell)$-system,
if every subset of size $\ell$ (\textit{$\ell$-set} for short) lies in  exactly one
edge of $H$. Replacing ``exactly one edge" by ``at most one edge", we have
 a  partial Steiner $(n,r,\ell)$-system.
In particular,  partial Steiner $(n,r,2)$-systems are also called
\textit{linear hypergraphs}, and Steiner $(n,3,2)$-systems are called
\textit{Steiner triple systems}. Let $\mathcal{H}_r(n, m)$ denote the set of $r$-graphs with $m$ edges,
  $\mathcal{L}_r^\ell(n, m)$ denote the set of partial Steiner $(n,r,\ell)$-systems
in $\mathcal{H}_r(n,m)$,  and  $\mathcal{L}_r^2(n,m)$
is specially denoted as $\mathcal{L}_r(n,m)$.

The \textit{uniform hypergraph process} $\mathbb{H}_r(n,m)$
is a Markov process with time running through the set $\{0,1,\cdots, \binom{n}{r}\}$.
It is the typical  random graph process $\mathbb{G}(n,m)$
 introduced by Erd\H{o}s and R\'{e}nyi when $r=2$~\cite{alan15}.
Similarly, the \textit {partial Steiner $(n,r,\ell)$-system process}
 begins with no edges on  vertex set $[n]$ at time $0$,
 all $r$-sets arrive one by one according to a uniformly
 chosen permutation, and each one is added if and only if it does not overlap any of the
 previously-added edges in $\ell$ or more vertices. In particular, it is the
  \textit {linear hypergraph process} when $\ell=2$. Let
  $\mathbb{L}_r^\ell(n, m)$ with $2\leqslant \ell\leqslant r-1$
  denote the $m$-th stage of the uniform partial Steiner
  $(n,r,\ell)$-system process,  and  $\mathbb{L}_r^2(n,m)$ is
   also denoted as $\mathbb{L}_r(n,m)$.

The hitting time of connectivity is a classic problem which has been extensively
studied in the theory of random graph processes.  Bollob\'{a}s and Thomason~\cite{bollo85} proved that,
with probability approaching to $1$ when $n\rightarrow\infty$ (\textit{w.h.p.} for short),
$m= \frac{n}{2}\log n$ is a sharp threshold of connectivity for $\mathbb{G}(n,m)$
and the very edge which links the last isolated vertex with another vertex makes the graph connected.
 Poole~\cite{poole15} proved the analogous result for $\mathbb{H}_r(n,m)$
 when $r\geqslant 3$ is a fixed integer, which means that $m= \frac{n}{r}\log n$
 is the hitting time  of connectivity for $\mathbb{H}_r(n,m)$.
The proofs in~\cite{bollo85,poole15} are due to the fact
that the $m$-th stage $\mathbb{H}_r(n,m)$  can be identified with the uniform random hypergraph
from $\mathcal{H}_r(n,m)$, and  behaves in
a similar fashion when $m$ equals or is close to
the expected number of edges of $\mathbb{H}_r(n, p)$, where
 a random $r$-graph $\mathbb{H}_r(n, p)$ is an $r$-graph
 on the vertex set $[n]$ and
 each $r$-set is an edge independently with probability $p$.

It might be surmised that the threshold  of connectivity for
$\mathbb{L}_r^\ell(n, m)$ is smaller than the one for $\mathbb{H}_r(n,m)$
because of its constraint on $r$-graphs. 
Let $\tau_c= \min\{m: \mathbb{L}_r^\ell(n, m)\ \text{is connected}\}$
 and $\tau_o= \min\{m: \mathbb{L}_r^\ell(n, m)\ \text{has no isolated vertices}\}$.
  These two properties are certainly monotone increasing properties,
then $\tau_c$ and $\tau_o$ are well-defined 
in $\mathbb{L}_r^\ell(n, m)$.
In this paper, for any fixed integers $r$  and $\ell$ with $2\leqslant\ell\leqslant r-1$,
we show that $\mathbb{L}_r^\ell(n, m)$ has the same threshold function
of connectivity with $\mathbb{H}_r(n,m)$, and $\mathbb{L}_r^\ell(n, m)$ also becomes
connected exactly at the moment when the last isolated vertex disappears.
\begin{theorem}\label{t1.6}
For any fixed integers $r$ and $\ell$ with $2\leqslant\ell\leqslant r-1$, w.h.p., $m= \frac{n}{r}\log n$
is a sharp threshold of connectivity for $\mathbb{L}_r^\ell(n, m)$ and
$\tau_{c}=\tau_{o}$ for $\mathbb{L}_r^\ell(n, m)$.
\end{theorem}

From  the proof of Theorem~\ref{t1.6}, we also have a corollary about
 the distribution of the number of isolated vertices in $\mathbb{L}_r^\ell(n, m)$ when
 $m= \frac{n}{r}\bigl(\log n+c_n)$ and $c_n\rightarrow c\in \mathbb{R}$.

\begin{corollary}\label{c1.8}
For any fixed integers $r$ and $\ell$ with $2\leqslant\ell\leqslant r-1$,
let $m= \frac{n}{r}\bigl(\log n+c_n)$ with $c_n\rightarrow c\in \mathbb{R}$. 
The number of isolated vertices in $\mathbb{L}_r^\ell(n, m)$ tends in distribution to
the Poisson distribution with mean $\exp[-c]$.
\end{corollary}

In order to prove Theorem~\ref{t1.6}, unlike the proofs in~\cite{bollo85,poole15}, we cannot work in
an analogue of the random hypergraph model $\mathbb{H}_r(n,p)$, since randomly-chosen
independent edges are very unlikely to generate a linear hypergraph. Instead, we will
rely on enumeration results in Theorem~\ref{t1.1} to Theorem~\ref{t1.4} below.

 Little is known about the enumeration of %
 distinct partial Steiner $(n,r,\ell)$-systems with a given number of edges.
Hasheminezhad and McKay~\cite{hashe20} obtained the asymptotic number of linear hypergraphs
with a given number of edges of each size, assuming a constant bound on the edge
size and $o(n^{\frac{4}{3}})$ edges. McKay and Tian~\cite{mckay18}
obtained the asymptotic enumeration formula for the set of $\mathcal{L}_r^\ell(n, m)$
as far as $m=o(n^{ \frac{3}{2}})$. Tian~\cite{tian2021} asymptotically determined
the number of linear multipartite hypergraphs  when the number of edges is $m=o(n^{ \frac{4}{3}})$.
In fact,  we can apply exactly the same approach to obtain an asymptotic formula for  $|\mathcal{L}_r^\ell(n, m)|$
when $3\leqslant\ell\leqslant r-1$ and $m=o(n^{ \frac{\ell+1}{2}})$.
It turns out that the proof is a little easier when $\ell\geqslant 3$,
as only one type of clusters needs to be considered, compared with four clusters
in the case $\ell=2$, see~\cite{mckay18}. Hence, the asymptotic expression when $\ell\geqslant 3$
is simpler than the corresponding expression when $\ell=2$, so the
statements of Theorem~\ref{t1.1} and Theorem~\ref{t1.3} cannot be combined.

Let $N_i= \binom{n-i}{r}$ for $0\leqslant i\leqslant n$
and $[x]_t=x(x-1)\cdots(x-t+1)$ for some positive integer $t$ be the falling factorial.
 The standard asymptotic notations $o$ and $O$ refer to $n\rightarrow\infty$.
 The floor and ceiling signs are omitted whenever they are not crucial.
\begin{theorem}{\rm{(\cite{mckay18}, Theorem 1.1)}}\label{t1.1}
For a fixed integer $r\geqslant 3$, let $m=m(n)$ be an integer
with $m=o(n^{ \frac{3}{2}})$. Then, as $n\rightarrow \infty$,
\begin{align*}
|\mathcal{L}_r(n, m)|
&={ \frac{ N_0^m}{m!}}\exp\biggl[- \frac{[r]_2^2[m]_2}{4n^2}-
\frac{[r]_2^3(3r^2-15r+20)m^3}{24n^4}+O\Bigl( \frac{m^2}{n^3}\Bigr)\biggr].
\end{align*}
\end{theorem}
\begin{theorem}\label{t1.3}
For  fixed integers $r$ and $\ell$ such that $3\leqslant \ell\leqslant r-1$,
let  $m=m(n)$ be an integer with $m=o(n^{ \frac{\ell+1}{2}})$.
Then, as $n\rightarrow \infty$,
\begin{align}
|\mathcal{L}_r^\ell(n, m)|
&={ \frac{N_0^m}{m!}}\exp\biggl[- \frac{[r]_{\ell}^2[m]_2}{2\ell!n^{\ell}}
+O\Bigl( \frac{m^2}{n^{\ell+1}}\Bigr)\biggr].\notag
\end{align}
\end{theorem}

As one application of Theorem~\ref{t1.1} and Theorem~\ref{t1.3},
using the same switching method as Theorem 1.4 in~\cite{mckay18} when $\ell=2$, we generalize
the probability that $H$ contains a given hypergraph as a subhypergraph
when $H\in \mathcal{L}_r^\ell(n, m)$ for $3\leqslant \ell\leqslant r-1$
chosen uniformly at random. At last, we have
\begin{theorem}\label{t1.4}
For fixed integers $r$ and $\ell$ such that $2\leqslant \ell\leqslant r-1$,
let $m=m(n)$ and $k=k(n)$ be integers with
 $m=o(n^{ \frac{\ell+1}{2}})$ and $k=o\bigl(
 \frac{n^{\ell+1}}{m^2}\bigr)$.
  Let $K=K(n)$ be a given $r$-graph in $\mathcal{L}_r^\ell(n, k)$ and
 $H\in \mathcal{L}_r^\ell(n, m)$ be chosen uniformly at random. Then, as
$n\rightarrow \infty$,
\begin{align*}
\mathbb{P}[K\subseteq H]= \frac{[m]_k}{ N_0^k}\exp\biggl[ \frac{[r]_{\ell}^2k^2}{2\ell!n^{\ell}}
+O\Bigl( \frac{k}{n^{\ell}}+ \frac{m^2k}{n^{\ell+1}}\Bigr)\biggr].
\end{align*}
\end{theorem}
\noindent The proof of Theorem~\ref{t1.4} when $3\leqslant \ell\leqslant r-1$
can be found   in the
appendix.


The remainder of the paper is structured as follows. Notation and  auxiliary results
are presented in Section~\ref{s:2}. In Section~\ref{s:3}, 
we consider Theorem~\ref{t1.3}, where the way to
prove them is a refinement of Theorem~1.1 in~\cite{mckay18}.
In Section~\ref{s:6},  we prove  Theorem~1.1. The last section concludes the work.
The proof of  Theorem~\ref{t1.4} is in the appendix.

\section{Notation and  auxiliary results}\label{s:2}

To state our results precisely, we need some definitions.
Let~$H$ be an~$r$-graph in $\mathcal{H}_r(n,m)$.
For~$U\subseteq [n]$, the \textit{codegree} of~$U$ in~$H$, denoted by~$\codeg({U})$,
is the number of edges of~$H$ containing~$U$.
In particular,~$\codeg({U})$ is the degree of~$v$ in~$H$
if~$U=\{v\}$ for~$v\in [n]$, denoted by~$\deg (v)$.
Given an integer $\ell$ with $2\leqslant\ell\leqslant r-1$, any~$\ell$-set
$\{x_1,\cdots,x_{\ell}\}\subseteq [n]$ in an edge~$e$ of~$H$
is called a \textit{link} of $e$ if $\codeg({x_1,\cdots,x_{\ell}})\geqslant 2$.
Two edges $e_i$ and $e_j$ in $H$ are called \textit{linked edges} if $|e_i\cap e_j|\geqslant\ell$.
As defined in~\cite{mckay18},
let $G_H$ be a simple graph whose vertices are the edges of~$H$,
with two vertices of $G$ adjacent iff the corresponding edges of~$H$ are linked.
An edge induced subgraph of $H$ corresponding to a non-trivial
component of $G_H$ is called a \textit{cluster} of $H$.

 Furthermore, for two positive-valued functions $f$, $g$
 on the variable $n$, we write $f\ll g$ to denote $\lim_{n\rightarrow\infty} f (n)/g(n) =0$,
 $f \sim g$ to denote $\lim_{n\rightarrow\infty} f (n)/g(n) =1$ and
 $f\lesssim g$ if and only if $\lim_{n\rightarrow\infty} \sup f (n)/g(n) \leqslant 1$ .
 For an event $A$ and a random variable $Z$ in an arbitrary probability space $(\Omega,\mathcal{F},\mathbb{P})$,
$\mathbb{P}[A]$ and $\mathbb{E}[Z]$  denote the probability of $A$ and the expectation
of $Z$.  An event is said to occur  with high probability (\textit {w.h.p.}  for short), if
the probability that it holds tends to 1 as $n\rightarrow\infty$.

In order to identify several events which have low probabilities
in the uniform probability space $\mathcal{H}_r(n,m)$ as
$m=o(n^{ \frac{\ell+1}{2}})$,
 the following two lemmas  will be useful.

\begin{lemma}[\cite{mckay18}, Lemma~2.1]\label{l2.1}
Let $t=t(n)\geqslant 1$ be an integer and
 $e_1,\ldots,e_{t}$ be distinct $r$-sets of $[n]$. For any given integer $r\geqslant 3$,
 let $H$ be an $r$-graph that is chosen uniformly at random from $\mathcal{H}_r(n,m)$.
Then the probability that $e_1,\ldots,e_t$ are  edges of $H$ is at most
$\bigl( \frac{m}{N}\bigr)^t$.
\end{lemma}

\begin{lemma}[\cite{mckay18}, Lemma~2.2]\label{l2.2}
Let $r$, $t$ and $\alpha$ be integers such that $r,t,\alpha=O(1)$ and $0\leqslant\alpha\leqslant rt$.
If a hypergraph $H$ is chosen uniformly at random from $\mathcal{H}_r(n,m)$,
then the expected number of sets of $t$ edges of $H$ whose union has
$rt-\alpha$ or fewer vertices is $O\bigl(m^t n^{-\alpha})$.
\end{lemma}

We will need the following Lemma~\ref{l2.6} from~\cite{green06}  to find the enumeration formula of
 $\mathcal{L}_r^\ell(n, m)$.

\begin{lemma}[\cite{green06}, Corollary~4.5]\label{l2.6}
Let $N\geqslant 2$ be an integer, and for $1\leqslant i\leqslant N$, let
real numbers $A(i)$, $B(i)$ be given such that
$A(i)\geqslant 0$ and $1-(i-1)B(i)\geqslant 0$. Define $A_1=\min_{i=1}^NA(i)$,
$A_2=\max_{i=1}^NA(i)$, $C_1=\min_{i=1}^NA(i)B(i)$
and $C_2=\max_{i=1}^NA(i)B(i)$. Suppose that there exists
a real number $\hat{c}$ with $0<\hat{c}< \frac{1}{3}$
such that $\max\{A/N,|C|\}\leqslant \hat{c}$ for all $A\in [A_1,A_2]$, $C\in[C_1,C_2]$.
Define $n_0$, $n_1$, $\cdots$, $n_N$ by $n_0=1$ and
\begin{align*}
\frac{n_i}{n_{i-1}}= \frac{A(i)}{i}(1-(i-1)B(i))
\end{align*}
for $1\leqslant i\leqslant N$, with the following interpretation: if $A(i)= 0$ or $1-(i-1)B(i)=0$, then $n_j=0$
for $i\leqslant j\leqslant N$. Then $\Sigma_1\leqslant \sum_{i=0}^{N}n_i\leqslant \Sigma_2$,
where $\Sigma_1=\exp[A_1- \frac{1}{2}A_1C_2]-(2e\hat{c})^N$ and
$\Sigma_2=\exp[A_2- \frac{1}{2}A_2C_1+ \frac{1}{2}A_2C_1^2]+(2e\hat{c})^N$.
\end{lemma}

\section{Enumeration of $\mathcal{L}_r^\ell(n, m)$}\label{s:3}

In this section, we first consider the asymptotic enumeration formula for
$\mathcal{L}_r^\ell(n, m)$ as
$3\leqslant\ell \leqslant r-1$ and $m=o(n^{ \frac{\ell+1}{2}})$ to extend the case of
$\ell=2$ and $m=o(n^{ \frac{3}{2}})$ in~\cite{mckay18}. It turns out that the proof
is a little easier when $\ell\geqslant 3$, as only one type of clusters needs to be considered,
compared with four clusters in the case $\ell=2$.
We remark that
the proof follows along the same line of~\cite{valelec,green06,mckay18,tian2021} and
 we are only giving the details here for the
sake of self-completeness.

Let $\mathbb{P}(n,r,\ell; m)$ denote the probability that an
$r$-graph $H\in \mathcal{H}_r(n,m)$ chosen uniformly
 at random is a partial Steiner $(n,r,\ell)$-system. Then
 $|\mathcal{L}_r^\ell(n, m)|={\binom{N}{m}} \mathbb{P}(n,r,\ell; m)$.
Our task is reduced to show that $\mathbb{P}(n,r,\ell; m)$
equals the later factor in Theorem~\ref{t1.3}.

Let $\mathcal{L}_r^{\ell,+}(n, m)\subset\mathcal{H}_r(n,m)$
be the set of $r$-graphs $H$ which satisfy the following
properties $\bf(a)$ and $\bf(b)$. We show that  the  expected number of  $r$-graphs
 in $\mathcal{H}_r(n,m)$ not satisfying the properties of
 $\mathcal{L}_r^{\ell,+}(n, m)$ is quite small such that
the removal of these  $r$-graphs from our main proof will lead
to some simplifications.

$\bf(a)$\  Every cluster of $H$ consists of two edges overlapping by $\ell$ vertices (Figure 1).

$\bf(b)$\  The number of clusters in $H$ is at most $M$, where
$M=\bigl\lceil\log n+ \frac{3^{\ell+2}r^{2\ell}m^2}{\ell!n^{\ell}}\bigr\rceil$.

\begin{figure}[!htb]
\centering
\includegraphics[width=0.5\textwidth]{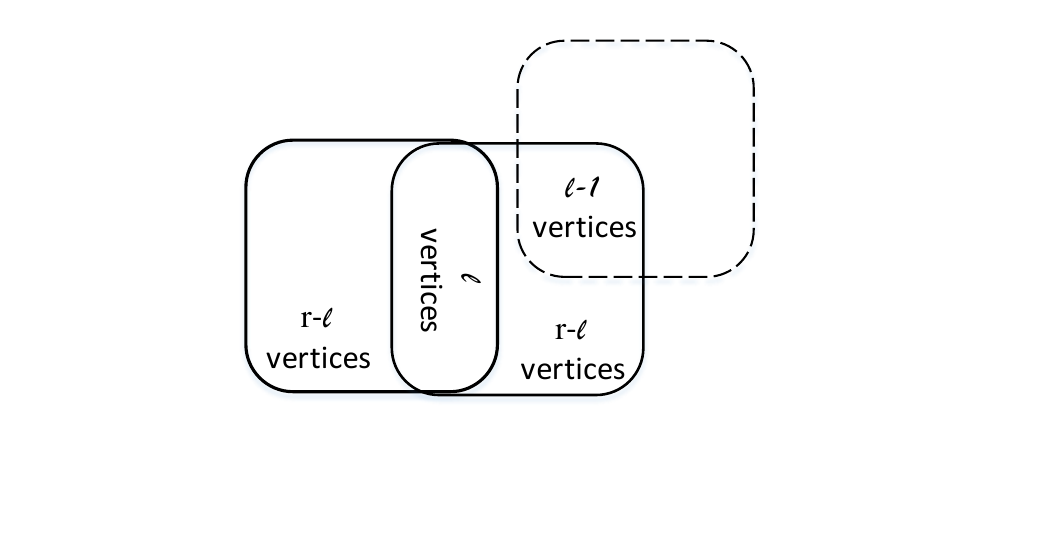}
\caption{The  cluster of $H\in\mathcal{L}_r^{\ell,+}(n,m)$.\label{fig:1}}
\end{figure}


\begin{lemma}\label{t3.2}
For any given integers $r$ and $\ell$ such that $3\leqslant\ell\leqslant r-1$, let $m=m(n)$ be integers with
 $m=o(n^{ \frac{\ell+1}{2}})$.
Then, as $n\rightarrow \infty$,
\begin{align*}
 \frac{|\mathcal{L}_r^{\ell,+}(n, m)|}
{|\mathcal{H}_r(n,m)|}=1-O\Bigl( \frac{m^2}{n^{\ell+1}}\Bigr).
\end{align*}
\end{lemma}

\begin{proof} Consider $H\in \mathcal{H}_r(n, m)$ chosen uniformly at random.
 We apply Lemma 2.2 several times
to show that $H$ satisfies the properties $\bf(a)$ and $\bf(b)$ with probability
$1-O( \frac{m^2}{n^{\ell+1}})$.

If two edges overlap by $\ell+1$ or more vertices, then they have at most $2r-\ell-1$ vertices
in total, which has probability $O( \frac{m^2}{n^{\ell+1}})$ by Lemma~2.2. Similarly
if there is a cluster of more than two edges, then three of those edges have at most
$3r-2\ell$ vertices in total, which has probability $O( \frac{m^3}{n^{2\ell}})=
O( \frac{m^2}{n^{\ell+1}})$ as $m=o(n^{ \frac{\ell+1}{2}})$ and $\ell\geqslant3$. Therefore,
$H$ satisfies the property $\bf(a)$ with probability $1-O( \frac{m^2}{n^{\ell+1}})$.


Note that if $\bf(a)$ holds, all clusters have two edges and no two clusters share an edge or
a link. Define the event  $\mathcal{E}=
\{{\rm{There\ are\ at\ least}}\ d\  \text{edge- and } \text{link-disjoint } {\rm{clusters\ in}}\ H\}$,
where $d=M+1$.  Let $\{x_1^i,\cdots,x_{\ell}^i\}\subseteq \binom{[n]}{\ell}$ be a set of
links with edges $e_i$ and $e_i'$ for $1\leqslant i\leqslant d$.
By Lemma~\ref{l2.1}, we have
\begin{align*}
\mathbb{P}[\mathcal{E}]=O\biggl(\binom{n}{r-\ell}^{2d}
\binom{\binom{n}{\ell}}{d}\Bigl( \frac{m}{N_0}\Bigr)^{2d}\biggr)
=O\biggl(\biggl( \frac{r^{2\ell}e m^2}{d\ell!n^{\ell}}\biggr)^{d}\biggr)
=O\Bigl( \frac{1}{n^{\ell+1}}\Bigr),
\end{align*}
where the last two equalities are true because
$d> \frac{3^{\ell+2}r^{2\ell}m^2}{\ell!n^{\ell}}$ and $d>\log n$.
The proof is complete on noting that the event ``$\bf(a)$ and $\bf(b)$"
is contained in the union of the event ``$\bf(a)$ holds" and ``$\mathcal{E}$ doesn't hold".
\end{proof}

For a nonnegative integer $t$, define
$\mathcal{L}_r^{\ell,+}(t)$ to
be the set of $r$-graphs $H\in \mathcal{S}_r^{\ell,+}(n, m)$
with exactly $t$ clusters and
we have  $|\mathcal{L}_r^{\ell,+}(n, m)|=\sum_{t=0}^{M}|\mathcal{L}_r^{\ell,+}(t)|$.
By Lemma~\ref{t3.2}, we have $|\mathcal{L}_r^{\ell,+}(n, m)|\neq0$ and
there exists $t$  such that $|\mathcal{L}_r^{\ell,+}(t)|\neq0$.
Note that
$\mathcal{L}_r^\ell(n, m)=\mathcal{L}_r^{\ell,+}(0)\neq\emptyset$,
then it follows that
\begin{align}\label{e3.1}
 \frac{1}{\mathbb{P}(n,r,\ell; m)}&=\Bigl(1-O\Bigl( \frac{m^2}{n^{\ell+1}}\Bigr)\Bigr)
 \sum_{t=0}^{M} \frac{|\mathcal{L}_r^{\ell,+}(t)|}{|\mathcal{L}_r^\ell(n, m)|}
 =\Bigl(1-O\Bigl( \frac{m^2}{n^{\ell+1}}\Bigr)\Bigr)
 \sum_{t=0}^{M} \frac{|\mathcal{L}^{\ell,+}(t)|}{|\mathcal{L}^{\ell,+}(0)|}.
\end{align}
In order to calculate the ratio
$|\mathcal{L}_r^{\ell,+}(t)|/|\mathcal{L}_r^{\ell,+}(0)|$
when $1\leqslant t\leqslant M$. We design switchings to find a relationship between the sizes of
 $\mathcal{L}_r^{\ell,+}(t)$ and $\mathcal{L}_r^{\ell,+}(t-1)$.
Let $H\in \mathcal{L}_r^{\ell,+}(t)$. A {\it forward switching} from $H$ is used to
reduce the number of clusters in $H$. Take any cluster $\{e,f\}$ and remove it from $H$.
Define $H_0$ with the same vertex set $[n]$ and the edge set $E(H_0)=E(H)\setminus \{e,f\}$.
Take any $r$-set from $[n]$ of which no $\ell$ vertices belong to the same edge of $H_0$ and
add it as a new edge. The graph is denoted by $H'$. Insert another new edge  at an $r$-set of $[n]$ again of which
no $\ell$ vertices belong to the same edge of $H'$. The resulting graph is denoted by $H''$.
The two new edges in forward switching may have at most $\ell-1$ vertices in common and
the operation reduces the number of clusters in $H$ by one.
A {\it reverse switching} is the reverse of a  forward switching.
A  reverse switching from $H''\in \mathcal{L}_r^{\ell,+}(t-1)$
is defined by sequentially removing  two edges of $H''$ not containing a link, then
choosing a $(2r-\ell)$-set $T$ from $[n]$ such that no $\ell$
vertices belong to any remaining edge of $H''$, then inserting two edges into $T$
such that they create a cluster.

\begin{lemma}\label{l5.3}
For any given integers $r$ and $\ell$ such that $3\leqslant\ell\leqslant r-1$, let  $m=m(n)$ be an integer with
$m=o(n^{ \frac{\ell+1}{2}})$. Let $t$ be a positive integer
with $1\leqslant t\leqslant M$.\\
$(a)$\ Let $H\in \mathcal{L}_r^{\ell,+}(t)$. The number of forward switchings for $H$ is
$t N_0^2
(1+O( \frac{m}{n^{\ell}}))$.

\noindent$(b)$\ Let $H''\in \mathcal{L}_r^{\ell,+}(t-1)$. The number of  reverse switchings for $H''$ is
$ \frac{(2r-\ell)!}{\ell!(r-\ell)!^2} \binom{m-2(t-1)}{2}
\binom{n}{2r-\ell}(1+O( \frac{m}{n^{\ell}}))$.
\end{lemma}

\begin{proof} $(a)$\ Let $H\in \mathcal{L}_r^{\ell,+}(t)$.
Let $\mathcal{R}(H)$ be the set of all forward switchings which can be applied to $H$.
There are exactly $t$ ways to choose a  cluster. The number of $r$-sets to insert the new edge is
at most $N_0$. From this we subtract the $r$-sets that have $\ell$ vertices belong to some
other edge of $H$, which is at most $ \binom{r}{\ell} m \binom{n-\ell}{r-\ell}=O( \frac{m}{n^\ell}) N_0$.
Thus, in each step of the forward switching, there are $N_0
(1+O( \frac{m}{n^{\ell}}))$ ways to choose the new edge and
we have $|\mathcal{R}(H)|=t N_0^2
(1+O( \frac{m}{n^{\ell}}))$.

$(b)$\ Conversely, suppose that $H''\in \mathcal{L}_r^{\ell,+}(t-1)$.
Similarly, let $\mathcal{R}'(H'')$ be the set of all reverse switchings for $H''$.
There are exactly $2\binom{m-2(t-1)}{2}$ ways to delete two edges in sequence such that
neither of them contain a link.  There are at most
$\binom{n}{2r-\ell}$ ways to choose a $(2r-\ell)$-set $T$ from $[n]$.
From this, we subtract the $(2r-\ell)$-sets
 that have $\ell$ vertices belong to some
other edge of $H''$, which is at most $\binom{r}{\ell} m \binom{n-\ell}{2r-2\ell}
=O( \frac{m}{n^\ell}) \binom{n}{2r-\ell}$.
For every $T$, there are $ \frac{1}{2} \binom{2r-\ell}{\ell} \binom{2r-2\ell}{r-\ell}$ ways to create a
 cluster in $T$.
 Thus, we have $|\mathcal{R}'(H'')|=\binom{2r-\ell}{\ell}
 \binom{2r-2\ell}{r-\ell} \binom{m-2(t-1)}{2} \binom{n}{2r-\ell}
(1+O( \frac{m}{n^\ell}))$.
\end{proof}

\begin{corollary}\label{c5.4}
With notation as above, for some $1\leqslant t\leqslant M$, the following hold: \\
$(a)$\ $|\mathcal{L}_r^{\ell,+}(t)|>0$ if and only if $m\geqslant 2t$.
\\
$(b)$\ Let $t'$ be the first value of $t\leqslant M$
such that $\mathcal{L}_r^{\ell,+}(t)=\emptyset$,
or $t'=M+1$ if no such value exists. 
Then, as $n\rightarrow\infty$, uniformly for $1\leqslant t< t'$,
\begin{align*}
 \frac{|\mathcal{L}_r^{\ell,+}(t)|}{|\mathcal{L}_r^{\ell,+}(t-1)|}=
 \binom{m-2(t-1)}{2} \frac{[r]_\ell^2}{\ell!t n^{\ell}}
\Bigl(1+O\Bigl( \frac{1}{n}\Bigr)\Bigr).
\end{align*}
\end{corollary}

\begin{proof} $(a)$\ Firstly, $m\geqslant 2t$ is a necessary condition for $|\mathcal{L}_r^{\ell,+}(t)|>0$.
By Lemma~\ref{t3.2}, there is some $0\leqslant \hat{t}\leqslant M$ such that
$\mathcal{L}_r^{\ell,+}(\hat{t})\neq\emptyset$. We can move $\hat{t}$ to $t$ by a sequence of forward
and reverse switchings while no greater than $M$.
Note that the values given in Lemma~\ref{l5.3} at each
step of this path are positive, we have $|\mathcal{L}_r^{\ell,+}(t)|>0$.

$(b)$\ By $(a)$,  if $\mathcal{L}_r^{\ell,+}(t)=\emptyset$, then
$\mathcal{L}_r^{\ell,+}(t+1),\cdots,\mathcal{L}_r^{\ell,+}(M) =\emptyset$.
By the definition of $t'$, the left hand ratio is well defined.
By Lemma~\ref{l5.3},  we complete the proof of $(b)$, where $O( \frac{m}{n^{\ell}})$
is absorbed into $O( \frac{1}{n})$ as $m=o(n^{ \frac{\ell+1}{2}})$ and $\ell\geqslant 3$.
\end{proof}

At last, by Lemma~\ref{l2.6}, we estimate
$\sum_{t=0}^{M} \frac{|\mathcal{L}_r^{\ell,+}(t)|}{|\mathcal{L}_r^{\ell,+}(0)|}$
in~\eqref{e3.1} to finish the proof of Theorem~\ref{t1.3}.

\begin{lemma}\label{l6.2}
For any given integers $r$ and $\ell$ such that $3\leqslant\ell\leqslant r-1$, let $m=m(n)$ be an
integer with $m=o(n^{ \frac{\ell+1}{2}})$.
With notation as above, as $n\rightarrow\infty$,
\begin{align*}
\sum_{t=0}^{M} \frac{|\mathcal{L}_r^{\ell,+}(t)|}{|\mathcal{L}_r^{\ell,+}(0)|}
=&\exp\biggl[ \frac{[r]_\ell^2[m]_2}{2\ell! n^{\ell}}+O\Bigl( \frac{m^2}{n^{\ell+1}}\Bigr)\biggr].
\end{align*}
\end{lemma}

\begin{proof} Let $t'$ be as defined in Lemma~\ref{c5.4}(b) and we have shown
$|\mathcal{L}_r^{\ell,+}(0)|=|\mathcal{L}_r^\ell(n, m)|\neq0$, then $t'\geqslant 1$. But if
$t'=1$, by Lemma~\ref{c5.4}(a), we have $m<2$ and the conclusion is obviously true.
In the following, suppose $t'\geqslant 2$. Define $n_{0},\cdots,n_{M}$ by $n_{0}=1$,
$n_{t}= |\mathcal{L}_r^{\ell,+}(t)|/|\mathcal{L}_r^{\ell,+}(0)|$
for $1\leqslant t<t'$ and $n_{t}=0$ for $t'\leqslant t\leqslant M$.
By Lemma~\ref{c5.4}(b), for $1\leqslant t< t'$, we have
\begin{equation}\label{e4.1}
\begin{aligned}[b]
 \frac{n_{t}}{n_{t-1}}&= \frac{1}{t}
\binom{m-2(t-1)}{2} \frac{[r]_\ell^2}{\ell! n^{\ell}}
\Bigl(1+O\Bigl( \frac{1}{n}\Bigr)\Bigr).
\end{aligned}
\end{equation}
For $1\leqslant t\leqslant M$, define
\begin{equation}\label{e3.32}
\begin{aligned}[b]
A(t)&=  \frac{[r]_\ell^2[m]_2}{2\ell! n^{\ell}}
\Bigl(1+O\Bigl( \frac{1}{n}\Bigr)\Bigr),\\
B(t)&=\begin{cases} \frac{2(2m-2t+1)}{m(m-1)},\text{for}\ 1\leqslant t<t';\\(t-1)^{-1},\text{otherwise}.\end{cases}
\end{aligned}
\end{equation}
As the equations shown in~\eqref{e4.1} and~\eqref{e3.32}, we further have $ \frac{n_{{t}}}{n_{t-1}}=
 \frac{A(t)}{t}(1-(t-1)B(t))$.

Following the notation of Lemma~\ref{l2.6}, we have
$A_1,A_2= \frac{[r]_\ell^2[m]_2}{2\ell! n^{\ell}}
(1+O( \frac{1}{n}))$. For $1\leqslant t< t'$, we have
$A(t)B(t)= \frac{[r]_\ell^2(2m-2t+1)}{\ell! n^{\ell}}
(1+O( \frac{1}{n}))$.
Thus, we have $A(t)B(t)= O( \frac{m}{n^\ell})$ for
$1\leqslant t< t'$. For the case $t'\leqslant t\leqslant M$ and $t'\geqslant 2$, by Lemma~\ref{c5.4}(a),
we have $2\leqslant m<2t$. As the equation shown in~\eqref{e3.32}, we also have
$A(t)B(t)= O( \frac{m}{n^\ell})$ for
$t'\leqslant t\leqslant M$. In both cases, we have $C_1,C_2=O( \frac{m}{n^\ell})$.
Note that  $|C|=o(1)$ for all $C\in[C_{1},C_{2}]$
 as $m=o(n^{ \frac{\ell+1}{2}})$.

Let $\hat{c}= \frac{1}{2(3^{\ell+2})}$, then $\max\{A/M,|C|\}\leqslant \hat{c}< \frac{1}{3}$ and
$(2e\hat{c})^{M}=O( \frac{1}{n^{\ell+1}})$ as $n\rightarrow\infty$.
Lemma~\ref{l2.6} applies to obtain $\sum_{t=0}^{M} \frac{|\mathcal{L}_r^{\ell,+}(t)|}{|\mathcal{L}_r^{\ell,+}(0)|}
=\exp\bigl[ \frac{[r]_\ell^2[m]_2}{2\ell! n^{\ell}}+
O\bigl( \frac{m^2}{n^{\ell+1}}\bigr)\bigr]$,
where $O( \frac{m^3}{n^{2\ell}})=O( \frac{m^2}{n^{\ell+1}})$
as $m=o(n^{ \frac{\ell+1}{2}})$.
\end{proof}

\begin{proof}[Proof of Theorem~\ref{t1.3}]By Lemma~\ref{l6.2}, as the equation shown in~\eqref{e3.1},
\begin{align*}
|\mathcal{L}_r^\ell(n, m)|&=\binom{N_0}{m} \exp\Bigl[ -\frac{[r]_\ell^2[m]_2}{2\ell! n^{\ell}}+
O\Bigl( \frac{m^2}{n^{\ell+1}}\Bigr)\Bigr]\\
&= \frac{N_0^m}{m!} \exp\Bigl[ -\frac{[r]_\ell^2[m]_2}{2\ell! n^{\ell}}+
O\Bigl( \frac{m^2}{n^{\ell+1}}\Bigr)\Bigr],
\end{align*}
where $\binom{N_0}{m}= \frac{N_0^m}{m!}
\exp\bigl[O\bigl( \frac{m^{2}}{N_0}\bigr)\bigr]
= \frac{N_0^m}{m!}\exp\bigl[O\bigl( \frac{m^{2}}{n^{\ell+1}}\bigr)\bigr]$. We complete the proof of Theorem~\ref{t1.3}.
\end{proof}

\begin{remark}

We also extend the probability that a
 random linear $r$-graph with $m=o(n^{ \frac{3}{2}})$ edges contains a given
  subhypergraph (Theorem~1.4 in~\cite{mckay18}),
  by similar discussions with appropriate modifications,
  to the case  $3\leqslant\ell \leqslant r-1$ and $m=o(n^{ \frac{\ell+1}{2}})$.
We show it in the Appendix for reference.
\end{remark}


\section{Connectivity for $\mathbb{L}_r^\ell(n, m)$}\label{s:6}

As one application of Theorem~\ref{t1.1} to
Theorem~\ref{t1.4},
we consider the hitting time of connectivity for partial Steiner $(n,r,\ell)$-system process
$\mathbb{L}_r^\ell(n, m)$ for any given integers $r$ and $\ell$ with $2\leqslant \ell\leqslant r-1$. 
It is clear that $\tau_{o}\leqslant \tau_{c}$.  Let
\begin{align}
m_{L}= \frac{n}{r}(\log n- \omega(n))\quad\text{and}\quad m_{R}= \frac{n}{r}(\log n+ \omega(n)),
\end{align}
where $\omega(n)\rightarrow\infty$ sufficiently slowly when $n\rightarrow\infty$
and taking $\omega(n)=\log\log n$ for convenience.

We prove our main result Theorem~\ref{t1.6} from a sequence of lemmas which we show next.

\begin{lemma}
Let $H$ be chosen from $\mathcal{L}_r^\ell(n, m)$
uniformly at random when $m=o(n^{ \frac{\ell+1}{2}})$,
and $v_1,\cdots,v_t\in [n]$ be $t$ distinct vertices for some fixed integer $t\geqslant 1$.
Then, as $n\rightarrow \infty$,
\begin{align*}
\mathbb{P}\bigl[\deg (v_1)=\cdots=\deg (v_t)=0\bigr]=\exp\Bigl[- \frac{trm}{n}
+O\Bigl( \frac{m}{n^2}+ \frac{m^2}{n^{\ell+1}}\Bigr)\Bigr].
\end{align*}
\end{lemma}

\begin{proof}\ By Theorem~\ref{t1.1} and Theorem~\ref{t1.3},
for one vertex $v\in [n]$, we have
\begin{align*}
\mathbb{P}\bigl[\deg (v)=0\bigr]&= \frac{|\mathcal{L}_r^\ell(n-1, m)|}{|\mathcal{L}_r^\ell(n, m)|}=
\frac{ N_1^m}{ N_0^m}\exp\biggl[O\Bigl( \frac{m^2}{n^{\ell+1}}\Bigr)\biggr]\\
&=\exp\biggl[- \frac{rm}{n}+O\Bigl( \frac{m}{n^2}+ \frac{m^2}{n^{\ell+1}}\Bigr)\biggr],
\end{align*}
where the last equality is true because $ \frac{ N_1}{N_0}
=\exp[- \frac{r}{n}+O( \frac{1}{n^2})]$.
Thus, for a  fixed integer $t\geqslant 1$,
\begin{align*}
\mathbb{P}\bigl[\deg (v_1)=\cdots=\deg (v_t)=0\bigr]= \frac{|\mathcal{L}_r^\ell(n-t, m)|}{|\mathcal{L}_r^\ell(n, m)|}
=\exp\biggl[- \frac{trm}{n}+O\Bigl( \frac{m}{n^2}+ \frac{m^2}{n^{\ell+1}}\Bigr)\biggr]
\end{align*}
to complete the proof of Lemma~4.1.
\end{proof}

\begin{lemma}
Let $H$ be chosen from $\mathcal{L}_r^\ell(n, m)$ uniformly at random.
\textit{W.h.p.} there are at most $2\log n$
isolated vertices in $H$ when $m=m_{L}$,
while \textit{w.h.p.} there are no isolated vertices in $H$ when $m=m_{R}$.
Thus,  $\tau_{o}\in [m_{L},m_{R}]$.
\end{lemma}

\begin{proof}\ Let $X_{m}$ be the number
of isolated vertices in $H$, where
 $m\in [m_{L},m_{R}]$. By Lemma~4.1, for any fixed integer $t\geqslant 1$,
 we have the $t$-th factorial moment of $X_{m}$ is
\begin{align}\label{e5.2}
\mathbb{E}[X_{m}]_t&=[n]_t\mathbb{P}\bigl[\deg (v_1)=\cdots=\deg (v_t)=0\bigr]\notag\\
&=[n]_t\exp\biggl[- \frac{trm}{n}
+O\Bigl( \frac{m}{n^2}+ \frac{m^2}{n^{\ell+1}}\Bigr)\biggr].
\end{align}

For $m=m_{R}$ and $t=1$, we have
\begin{align*}
\mathbb{E}[X_{m_{R}}]&=n\exp\biggl[- \frac{rm_R}{n}
+O\Bigl( \frac{m_R}{n^2}+ \frac{m_R^2}{n^{\ell+1}}\Bigr)\biggr]\\
&=\exp\Bigl[-\omega(n)+O\Bigl( \frac{\log n}{n}+ \frac{\log^2 n}{n^{\ell-1}}\Bigr)\Bigr]\\
&\rightarrow0
\end{align*}
when $n\rightarrow\infty$. Thus, \textit{w.h.p.}, there are no isolated
vertices in $H$ when $m=m_R$. 

For $m=m_{L}$ and $t=1$, we have
\begin{align}\label{e44.3}
\mathbb{E}[X_{m_{L}}]=
\exp\Bigl[\omega(n)+O\Bigl( \frac{\log n}{n}+ \frac{\log^2 n}{n^{\ell-1}}\Bigr)\Bigr]\rightarrow\infty
\end{align}
when $n\rightarrow\infty$.
For $m=m_{L}$ and $t=2$, using the equations in~\eqref{e5.2} and~\eqref{e44.3},
\begin{align*}
\mathbb{E}[X_{m_L}]_2&=[n]_2\exp\biggl[- \frac{2rm_L}{n}
+O\Bigl( \frac{m_L}{n^2}+ \frac{m_L^2}{n^{\ell+1}}\Bigr)\biggr]\\
&= \frac{[n]_2}{n^2}\exp\biggl[2\omega(n)
+O\Bigl( \frac{m_L}{n^2}+ \frac{m_L^2}{n^{\ell+1}}\Bigr)\biggr]\\
&\sim \mathbb{E}^2[X_{m_L}].
\end{align*}
Then, $\mathbb{V}[X_{m_L}]\sim \mathbb{E}[X_{m_L}]$.
By Chebyshev's inequality and $\mathbb{E}[X_{m_{L}}]\rightarrow\infty$
shown in~\eqref{e44.3}, $\mathbb{P}[|X_{m_L}-\mathbb{E}[X_{m_L}]|\geqslant
\mathbb{E}[X_{m_L}]]\leqslant \mathbb{V}[X_{m_L}]/\mathbb{E}^2[X_{m_L}]\sim
1/\mathbb{E}[X_{m_L}]\rightarrow 0$.
Thus, \textit{w.h.p.}, we have at most $2\log n$ isolated vertices in $H$ when $m=m_{L}$ 
because $X_{m_L}$ is concentrated around $\exp[\omega(n)]=\log n$ when $\omega(n)=\log\log n$.
\end{proof}
\begin{lemma}
If $H$ is chosen uniformly at random from $\mathcal{S}(n,r,\ell; m_L)$, then \textit{w.h.p.}
$H$ has at most $2\log n$ isolated vertices and all remaining vertices are in a giant component.
\end{lemma}

\begin{proof}\ Suppose that $H$ is chosen from $\mathcal{L}_r^\ell(n, m_L)$
uniformly at random. By Lemma~4.2, 
we only prove that \textit{w.h.p.} all non-isolated vertices in $H$ belong to a giant component.

For any nonnegative integers $k$  and $h$, 
let $Y_{k,h}$ be the number of components on
$k$ vertices with exactly $h$ edges in $H$.
By symmetry, we can assume $k\in [r, \frac{n}{2}]$.
On one hand, we have $h=h(k)\geqslant \frac{k-1}{r-1}$ because
every component is connected and $ \frac{k-1}{r-1}$ is the number of edges in a hypertree; on the other hand,
$h=h(k)\leqslant \min\bigl\{m_L, \binom{k}{\ell}/\binom{r}{\ell}\bigr\}$
because it is also a partial Steiner $(n,r,\ell)$-system. Let 
\begin{align}
h_{\rm{min}}= \frac{k-1}{r-1}\quad{\text{and}}\quad
h_{\rm{max}}=\min\Bigl\{m_L, \binom{k}{\ell}/\binom{r}{\ell}\Bigr\}.
\end{align}

Fix  $k\in [r, \frac{n}{2}]$ and choose some $k$-set on $[n]$,
then the probability that the $k$-set contains exactly $h$ edges
is at most $|\mathcal{L}_r^\ell(k, h)|
\cdot|\mathcal{L}_r^\ell(n-k, m_L-h)|/|\mathcal{L}_r^\ell(n, m_L)|$.
It is clear that
\begin{align}\label{e5.5}
\mathbb{E}[Y_{k,h}]
&\leqslant \frac{ \binom{n}{k}|\mathcal{L}_r^\ell(k, h)|
\cdot|\mathcal{L}_r^\ell(n-k, m_L-h)|}{|\mathcal{L}_r^\ell(n, m_L)|}.
\end{align}
Let $Y_k=\sum_{h_{\rm{min}}\leqslant h\leqslant h_{\rm{max}}} Y_{k,h}$.
We will prove
\begin{align}
\sum_{r\leqslant k\leqslant \frac{n}{2}}\mathbb{E}[Y_{k}]=
\sum_{r\leqslant k\leqslant \frac{n}{2}}\sum_{h_{\rm{min}}\leqslant h\leqslant h_{\rm{max}}}\mathbb{E}[Y_{k,h}]\rightarrow0
\end{align}
to show that the remaining vertices in $H$ \textit{w.h.p.} are all in a giant component.

Define
\begin{align}\label{e22.7}
I_1=\Bigl[r, \frac{n}{\log n}\Bigr]\ \text{and}\
I_2=\Bigl[ \frac{n}{\log n}, \frac{n}{2}\Bigr].
\end{align}
We firstly show $\sum_{k\in I_1}\mathbb{E}[Y_{k}]\rightarrow0$ from Claim 1 to Claim 4, then
$\sum_{k\in I_2}\mathbb{E}[Y_{k}]\rightarrow0$ from Claim 5 to Claim 8. 

\bigskip
\noindent{\bf Claim 1}.~~For any $k\in I_1$ and $h_{\rm{min}}\leqslant h\leqslant h_{\rm{max}}$,
\begin{align*}
 \bigl|\mathcal{L}_r^\ell(n-k, m_L-h)\bigr|
\sim
\frac{ N_k^{m_L-h}}{(m_L-h)!}
\exp\biggl[- \frac{[r]_\ell^2 [m_L-h]_2}{2\ell!(n-k)^\ell}\biggr].
\end{align*}

\begin{proof}[Proof of Claim 1]\ Note that $n-k\rightarrow\infty$ when
$k\in I_1$. If $\ell=2$, by Theorem~\ref{t1.1}, we have
\begin{equation*}
\begin{aligned}[b]
\bigl |\mathcal{L}_r^\ell(n-k, m_L-h)\bigr|&
=  \frac{N_k^{m_L-h}}{(m_L-h)!}
\exp\biggl[- \frac{[r]_2^2 [m_L-h]_2}{4(n-k)^2}+O\biggl( \frac{(m_L-h)^3}{(n-k)^4}+ \frac{(m_L-h)^2}{(n-k)^3}\biggr)\biggr]\\
&\sim \frac{ N_k^{m_L-h}}{(m_L-h)!}
\exp\biggl[- \frac{[r]_2^2 [m_L-h]_2}{4(n-k)^2}\biggr],
\end{aligned}
\end{equation*}
where the last approximate equality is true because  $m_{L}= \frac{n}{r}(\log n- \omega(n))$
and $k\in I_1$.

If $3\leqslant \ell\leqslant r-1$, by Theorem~\ref{t1.3}, we similarly also have
\begin{equation*}
\begin{aligned}[b]
\bigl |\mathcal{L}_r^\ell(n-k, m_L-h)\bigr|&
=  \frac{ N_k^{m_L-h}}{(m_L-h)!}
\exp\biggl[- \frac{[r]_\ell^2 [m_L-h]_2}{2\ell!(n-k)^\ell}+O\Bigl( \frac{(m_L-h)^2}{(n-k)^{\ell+1}}\Bigr)\biggr]\\
&\sim \frac{ N_k^{m_L-h}}{(m_L-h)!}
\exp\biggl[- \frac{[r]_\ell^2 [m_L-h]_2}{2\ell!(n-k)^\ell}\biggr].
\end{aligned}
\end{equation*}
\end{proof}

\bigskip
\noindent{\bf Claim 2}.~~For any $k\in I_1$ and $h_{\rm{min}}\leqslant h\leqslant h_{\rm{max}}$,
\begin{align*}
 \bigl|\mathcal{L}_r^\ell(k, h)\bigr|\cdot\bigl|\mathcal{L}_r^\ell(n-k, m_L-h)\bigr|
&< \frac{ \binom{k}{r}^{h_{{\rm min}}}}{h_{{\rm min}}!}
\frac{ N_k^{m_L-h_{{\rm min}}}}{(m_L-h_{{\rm min}})!}
\exp\biggl[- \frac{[r]_\ell^2 [m_L-h_{{\rm min}}]_2}{2\ell!(n-k)^\ell}\biggr].
\end{align*}

\begin{proof}[Proof of Claim 2]\ Firstly, it is clear that $|\mathcal{L}_r^\ell(k, h)|
\leqslant \binom{ \binom{k}{r}}{h}$.
By Claim 1, we have
\begin{equation}\label{e44.7}
\begin{aligned}[b]
 \bigl|\mathcal{L}_r^\ell(k, h)\bigr|\cdot\bigl|\mathcal{L}_r^\ell(n-k, m_L-h)\bigr|
&< \frac{\binom{k}{r}^{h}}{h!} \frac{ N_k^{m_L-h}}{(m_L-h)!}
\exp\biggl[- \frac{[r]_\ell^2 [m_L-h]_2}{2\ell!(n-k)^\ell}\biggr].
\end{aligned}
\end{equation}

Let
\begin{align*}
g_1(h)= \frac{\binom{k}{r}^{h}}{h!} \frac{N_k^{m_L-h}}{(m_L-h)!}
\exp\biggl[- \frac{[r]_\ell^2 [m_L-h]_2}{2\ell!(n-k)^\ell}\biggr].
\end{align*}
Note that
\begin{align*}
 \frac{g_1(h+1)}{g_1(h)}&= \frac{\binom{k}{r}}{h+1} \frac{m_L-h}{N_k}
 \exp\biggl[ \frac{[r]_\ell^2 [m_L-h]_2}{2\ell!(n-k)^\ell}- \frac{[r]_\ell^2 [m_L-h-1]_2}{2\ell!(n-k)^\ell}\biggr]\\
 &= \frac{\binom{k}{r}}{h+1} \frac{m_L-h}{N_k}\exp\Bigl[O\Bigl( \frac{m_L}{(n-k)^\ell}\Bigr)\Bigr].
\end{align*}
By $h\geqslant h_{{\rm min}}\geqslant \frac{k}{r}$,
 $m_{L}= \frac{n}{r}(\log n- \omega(n))$ and $\ell\geqslant 2$, we further have
\begin{align*}
 \frac{g_1(h+1)}{g_1(h)}&
 < \frac{\binom{k}{r}}{k} \frac{n\log n}{ N_k}\exp\Bigl[O\Bigl( \frac{m_L}{(n-k)^\ell}\Bigr)\Bigr]
 =O\Bigl( \frac{k^{r-1}\log n}{n^{r-1}}\Bigr),
\end{align*}
which implies that $ \frac{g_1(h+1)}{g_1(h)}\rightarrow 0$ when
$k\in I_1$ in~\eqref{e22.7} and $r\geqslant 3$. Thus,
 $g_1(h)$ is decreasing in $h$. 
Using the equation~\eqref{e44.7} with $h=h_{\rm{min}}$,
we complete the proof of Claim 2.
\end{proof}

\bigskip
\noindent{\bf Claim 3}.~~ For any $k\in I_1$ and $h_{{\rm min}}\leqslant h\leqslant h_{{\rm max}}$,
\begin{align*}
\mathbb{E}[Y_{k,h}]
< \frac{k^{rh_{{\rm min}}}(\log n)^{k+h_{{\rm min}}}}
{r^{h_{{\rm min}}}k!h_{{\rm min}}!n^{k-1}}\exp\Bigl[ \frac{krh_{{\rm min}}}{n}\Bigr].
\end{align*}

\begin{proof}[Proof of Claim 3]\ By Theorem~\ref{t1.1} and Theorem~\ref{t1.3},
for $2\leqslant \ell\leqslant r-1$
and $m_L= \frac{n}{r}(\log n-\omega(n))$,
we also have
\begin{align*}
\bigl|\mathcal{L}_r^\ell(n, m_L)\bigr|\sim  \frac{ N_0^{m_L}}{m_L!}
\exp\biggl[- \frac{[r]_\ell^2 [m_L]_2}{2\ell!n^\ell}\biggr].
\end{align*}
Using the equations shown in~\eqref{e5.5} and Claim 2, it follows that
\begin{align*}
\mathbb{E}[Y_{k,h}]&<
\frac{ \binom{n}{k} \frac{ \binom{k}{r}^{h_{{\rm min}}}}{h_{{\rm min}}!}
\frac{ N_k^{m_L-h_{{\rm min}}}}{(m_L-h_{{\rm min}})!}
\exp\Bigl[- \frac{[r]_\ell^2 [m_L-h_{{\rm min}}]_2}{2\ell!(n-k)^\ell}\Bigr]}
{ \frac{ N_0^{m_L}}{m_L!}\exp\Bigl[- \frac{[r]_\ell^2 [m_L]_2}{2\ell!n^\ell}\Bigr]}\notag\\
&< \frac{ m_L^{h_{{\rm min}}}\binom{n}{k} \binom{k}{r}^{h_{{\rm min}}}
 N_k^{m_L-h_{{\rm min}}}}{ h_{{\rm min}}! N_0^{m_L}}\exp\biggl[- \frac{[r]_\ell^2 [m_L-h_{{\rm min}}]_2}{2\ell!(n-k)^\ell}+
\frac{[r]_\ell^2 [m_L]_2}{2\ell!n^\ell}\biggr]\\
&<\frac{ m_L^{h_{{\rm min}}}\binom{n}{k} \binom{k}{r}^{h_{{\rm min}}}
N_k^{m_L-h_{{\rm min}}}}{ h_{{\rm min}}! N_0^{m_L}},
\end{align*}
where the last inequality is true because $ \frac{k\ell}{n}-\frac{2h_{{\rm min}}}{m_L}>0$
when $k\in I_1$ and $m_L= \frac{n}{r}(\log n-\omega(n))$, and hence
\begin{align*}
\exp\biggl[- \frac{[r]_\ell^2 [m_L-h_{{\rm min}}]_2}{2\ell!(n-k)^\ell}+
\frac{[r]_\ell^2 [m_L]_2}{2\ell!n^\ell}\biggr]\sim\exp\biggl[- \frac{[r]_\ell^2 [m_L]_2}{2\ell!n^\ell}\biggl( \frac{k\ell}{n}-\frac{2h_{{\rm min}}}{m_L}\biggr)\biggr]< 1.
\end{align*}
Note that $ N_0\sim \frac{n^{r}}{r!}$,
$ N_k\sim
\frac{(n-k)^r}{r!}$,
$ \binom{n}{k}\sim \frac{n^k}{k!}$ and
$ \binom{k}{r}\leqslant \frac{k^r}{r!}$ when $k\in I_1$ and $n\rightarrow\infty$, we further have
\begin{align}\label{e5.81}
\mathbb{E}[Y_{k,h}]&< \frac{m_L^{h_{{\rm min}}}k^{rh_{{\rm min}}}(n-k)^{r(m_L-h_{{\rm min}})}}
{k!h_{{\rm min}}!n^{rm_L-k}}.
\end{align}
Substituting $m_L= \frac{n}{r}(\log n- \log\log n)$ into the above equation, we have $m_L^{h_{{\rm min}}}<
\frac{n^{h_{{\rm min}}}}{r^{h_{{\rm min}}}}(\log n)^{h_{{\rm min}}}$  and  
\begin{align*}
(n-k)^{r(m_L-h_{{\rm min}})}
&\leqslant n^{r(m_L-h_{{\rm min}})}\exp\biggl[- \frac{kr(m_L-h_{{\rm min}})}{n}\biggr]\\
&= \frac{n^{r(m_L-h_{{\rm min}})} \log^k n}{n^k}\exp\biggl[ \frac{krh_{{\rm min}}}{n}\biggr].
\end{align*}
Thus, by $(r-1)h_{\rm{min}}= k-1$, the equation in (4.9) is reduced to
\begin{align*}
\mathbb{E}[Y_{k,h}]
&<\frac{k^{rh_{{\rm min}}}(\log n)^{k+h_{{\rm min}}}}{r^{h_{{\rm min}}}k!h_{{\rm min}}!n^{k-1}}\exp\biggl[ \frac{krh_{{\rm min}}}{n}\biggr].
\end{align*}
\end{proof}

\bigskip
\noindent{\bf Claim 4}.~~$\sum_{k\in I_1}\mathbb{E}[Y_{k}]\rightarrow0$. 

\begin{proof}[Proof of Claim 4]\
By the equations in (4.4), Claim 3 and $m_L< \frac{n}{r}\log n$,  it follows that
\begin{align*}
\mathbb{E}[Y_k]&= \sum_{h_{{\rm min}}\leqslant h\leqslant h_{{\rm max}}}\mathbb{E}[Y_{k,h}]<
\frac{k^{rh_{{\rm min}}}(\log n)^{k+h_{{\rm min}}+1}}{r^{h_{{\rm min}}+1}k!h_{{\rm min}}!n^{k-2}}\exp\biggl[ \frac{krh_{{\rm min}}}{n}\biggr].
\end{align*}
Since $h_{\rm{min}}= \frac{k-1}{r-1}\geqslant \frac{k}{r}$ when $k\in I_1$, we have
$h_{{\rm min}}!\geqslant ( \frac{h_{{\rm min}}}{e})^{h_{{\rm min}}}
\geqslant ( \frac{k}{re})^{h_{{\rm min}}}$. By $k!\geqslant ( \frac{k}{e})^k$
and $(r-1)h_{\rm{min}}= k-1$,
we also have 
\begin{align}\label{e5.55}
\mathbb{E}[Y_k]&<
\frac{k^{rh_{{\rm min}}}(\log n)^{k+h_{{\rm min}}+1}}{rk^{k+h_{{\rm min}}}n^{k-2}}
\exp\biggl[ k+ h_{\rm{min}}+ \frac{krh_{\rm{min}}}{n}\biggr]\notag\\
&= \frac{(\log n)^{k+ \frac{k-1}{r-1}+1}  }{rk  n^{k- 2}}
\exp\biggl[ k+ \frac{k-1}{r-1}+ \frac{kr(k-1)}{n(r-1)}\biggr].
\end{align}

Let
\begin{align*}
g_2(k)= \frac{(\log n)^{k+ \frac{k-1}{r-1}+1}}{  n^{k- 2}}
\exp\biggl[ k+ \frac{k-1}{r-1}+ \frac{kr(k-1)}{n(r-1)}\biggr].
\end{align*}
For all $k\in I_1$,
\begin{align*}
 \frac{g_2(k+1)}{g_2(k)}&= \frac{(\log n)^{1+ \frac{1}{r-1}}}{n}\exp\biggl[ 1+ \frac{1}{r-1}+ \frac{2kr}{n(r-1)}\biggr]\\
 &\leqslant \frac{(\log n)^{1+ \frac{1}{r-1}}}{n}\exp\biggl[ 1+ \frac{1}{r-1}+ \frac{2r}{(r-1)\log n}\biggr]\\
 &\rightarrow 0.
\end{align*}
We have $g_2(k)$ is decreasing in $k$ for all $k\in I_1$ and $n\rightarrow\infty$, which implies
\begin{align*}
g_2(k)\leqslant g_2(r)= \frac{(\log n)^{r+ 2}}{  n^{r- 2}}
\exp\biggl[ r+1+ \frac{r^2}{n}\biggr].
\end{align*}
Hence by the equation in~\eqref{e5.55}, we have
\begin{align*}
\mathbb{E}[Y_k]< \frac{1}{rk}g_2(k)
\leqslant \frac{(\log n)^{r+ 2}}{  rkn^{r- 2}}
\exp\biggl[ r+1+ \frac{r^2}{n}\biggr].
\end{align*}
Finally, by $\sum_{k\in I_1} k^{-1}=O(\log n)$, we have
$\sum_{k\in I_1}\mathbb{E}[Y_k]=O(n^{2-r}(\log n)^{ r+3})\rightarrow0$ for $r\geqslant 3$
to complete the proof of Claim 4.
\end{proof}

In the following, we show $\sum_{k\in I_2}\mathbb{E}[Y_{k}]\rightarrow0$
from Claim 5 to Claim 8.  Since
$k\rightarrow\infty$ and $n-k\rightarrow\infty$ when $k\in I_2$,
by Theorem~\ref{t1.1} and Theorem~\ref{t1.3},
for all $2\leqslant \ell \leqslant r-1$ and $h_{{\rm min}}\leqslant h\leqslant h_{{\rm max}}$,
we have
\begin{align}\label{e5.9}
\bigl|\mathcal{L}_r^\ell(k, h)\bigr|&\sim{ \frac{ {\binom{k}{r}}^h}{h!}}
\exp\biggl[- \frac{[r]_\ell^2[h]_2}{2\ell!k^\ell}\biggr],\\
\bigl|\mathcal{L}_r^\ell(n-k, m_L-h)\bigr|&\sim{ \frac{ {N_k}^{m_L-h}}{(m_L-h)!}}
\exp\biggl[- \frac{[r]_\ell^2[m_L-h]_2}{2\ell!(n-k)^\ell}\biggr].
\end{align}

\bigskip
\noindent{\bf Claim 5}.~~ For any $k\in I_2$ and $h\geqslant \frac{m_L}{2}$,
$\bigl|\mathcal{L}_r^\ell(k, h)\bigr|\cdot\bigl|\mathcal{L}_r^\ell(n-k, m_L-h)\bigr|$
is decreasing in $h$.

\begin{proof}[Proof of Claim 5]\ Using the equations shown in~\eqref{e5.9} and (4.12),
we have
\begin{align*}
& \frac{|\mathcal{L}_r^\ell(k, h+1)|\cdot|\mathcal{L}_r^\ell(n-k, m_L-h-1)|}
{|\mathcal{L}_r^\ell(k, h)|\cdot|\mathcal{L}_r^\ell(n-k, m_L-h)|}\\
&\sim \frac{\binom{k}{r}}{(h+1)} \frac{m_L-h}{ N_k}
\exp\biggl[- \frac{[r]_\ell^2h}{\ell!k^\ell}+ \frac{[r]_\ell^2(m_L-h-1)}{\ell!(n-k)^\ell}\biggr].
\end{align*}
Since $ \frac{h}{k^\ell}> \frac{m_L-h-1}{(n-k)^\ell}$ and  $ \frac{m_L-h}{(h+1)}<1$
when $k\leqslant \frac{n}{2}$ and $h\geqslant \frac{m_L}{2}$, we have
$  \frac{m_L-h}{(h+1)}
\exp\bigl[- \frac{[r]_\ell^2h}{\ell!k^\ell}+ \frac{[r]_\ell^2(m_L-h-1)}{\ell!(n-k)^\ell}\bigr]<1$
and
\begin{align}
& \frac{|\mathcal{L}_r^\ell(k, h+1)|\cdot|\mathcal{L}_r^\ell(n-k, m_L-h-1)|}
{|\mathcal{L}_r^\ell(k, h)|\cdot|\mathcal{L}_r^\ell(n-k, m_L-h)|}
< \frac{ \binom{k}{r}}{N_k}\leqslant 1
\end{align}to complete the proof of Claim 5.
\end{proof}

\bigskip
\noindent{\bf Claim 6}.~~For any $k\in I_2$ and $h\leqslant \frac{m_L}{2}$,
\begin{align*}
  \exp\biggl[- \frac{[r]_\ell^2[h]_2}{2\ell!k^\ell}- \frac{[r]_\ell^2[m_L-h]_2}
  {2\ell!(n-k)^\ell}+ \frac{[r]_\ell^2[m_L]_2}{2\ell!n^\ell}\biggr]<1.
\end{align*}

\begin{proof}[Proof of Claim 6]\  Suppose $h=t_1m_L$ and $k=t_2n$.
Since $m_L= \frac{n}{r}(\log n-\omega(n))$ in (4.1), $ \frac{k-1}{r-1}\leqslant h\leqslant \frac{m_L}{2}$
in (4.4) and $k\in I_2= [ \frac{n}{\log n}, \frac{n}{2}]$ in (4.7), we have
$ \frac{1}{\log n}\leqslant t_2\leqslant \frac{1}{2}$, $ \frac{t_2}{\log n}< t_1\leqslant \frac{1}{2}$,
 $[h]_2\sim h^2$, $[m_L-h]_2\sim (m_L-h)^2$ and $[m_L]_2\sim m_L^2$ because
 $h\rightarrow\infty$ and $m_L-h\rightarrow\infty$ when $n\rightarrow\infty$.

Thus, $ \frac{[h]_2}{k^\ell}+ \frac{[m_L-h]_2}
  {(n-k)^\ell}> \frac{[m_L]_2}{n^\ell}$ is equivalent to $\frac{t_1^2}{t_2^\ell}+
\frac{(1-t_1)^2}{(1-t_2)^\ell}>1$, which is clearly true because
\begin{align*}
 \frac{t_1^2}{t_2^\ell}+
\frac{(1-t_1)^2}{(1-t_2)^\ell}\geqslant \frac{t_1^2}{t_2^2}+
\frac{(1-t_1)^2}{(1-t_2)^2}> 1
\end{align*} when $\ell\geqslant 2$, $0< t_1< 1$
and  $ 0< t_2< 1$.
\end{proof}

\bigskip
\noindent{\bf Claim 7}.~~ For any $k\in I_2$ and $h\leqslant \frac{m_L}{2}$,
\begin{align*}
\mathbb{E}[Y_{k,h}]=
O\Bigl(n^{- 1}2^n\bigl(\log n\bigr)^{ \frac{(1-r)m_L}{\log n}}\Bigr).
\end{align*}

\begin{proof}[Proof of Claim 7]\ 
By Theorem~\ref{t1.1} and Theorem~\ref{t1.3}, for all $2\leqslant \ell\leqslant r-1$, we have
\begin{align}\label{e5.10}
\bigl|\mathcal{L}_r^\ell(n, m_L)\bigr|
\sim  \frac{ N_0^{m_L}}{m_L!}
\exp\biggl[-  \frac{[r]_\ell^2[m_L]_2}{2\ell!n^\ell}\biggr].
\end{align}
Using the equations  in (4.5), ~\eqref{e5.9}, (2.12) and~\eqref{e5.10}, we have
\begin{align}\label{e55.6}
& \mathbb{E}[Y_{k,h}]\leqslant \frac{\binom{n}{n/2}\bigl|\mathcal{L}_r^\ell(k, h)\bigr|\cdot\bigl|\mathcal{L}_r^\ell(n-k, m_L-h)\bigr|}
{\bigl|\mathcal{L}_r^\ell(n, m_L)\bigr|}\notag\\
&\sim  \frac{ \binom{n}{n/2}\binom{k}{r}^{h} N_k^{m_L-h}}{ N_0^{m_L}} \frac{m_L!}{h!(m_L-h)!}
\exp\biggl[- \frac{[r]_\ell^2[h]_2}{2\ell!k^\ell}- \frac{[r]_\ell^2[m_L-h]_2}
{2\ell!(n-k)^\ell}+ \frac{[r]_\ell^2[m_L]_2}{2\ell!n^\ell}\biggr]\notag\\
&< \frac{ \binom{n}{n/2} \binom{k}{r}^{h} N_k^{m_L-h}}{N_0^{m_L}} \frac{m_L!}{h!(m_L-h)!},
\end{align}
where the last inequality is true  by Claim 6.

Let
\begin{align*}
g_3(k,h)=  \frac{\binom{k}{r}^{h}N_k^{m_L-h}}
{N_0^{m_L}} \frac{m_L!}{h!(m_L-h)!}.
\end{align*}
Firstly, for $k\in I_2$ and $n\rightarrow\infty$, we have $\binom{k}{r}\sim \frac{k^r}{r!}$,
$ N_k\sim \frac{(n-k)^r}{r!}$ and $ N_0\sim \frac{n^r}{r!}$.
Secondly, for any $h\in [ \frac{k-1}{r-1}, \frac{m_L}{2}]$,
we have $h\rightarrow\infty$ and $m_L-h\rightarrow\infty$.
By Stirling's formula,
 \begin{align*}
 h!&\sim \sqrt{2\pi h}\Bigl( \frac{h}{e}\Bigr)^h,\\
 (m_L-h)!&\sim \sqrt{2\pi (m_L-h)}\Bigl( \frac{m_L-h}{e}\Bigr)^{m_L-h},\\
m_L!&\sim \sqrt{2\pi m_L}\Bigl( \frac{m_L}{e}\Bigr)^{m_L},
\end{align*}
we further have
\begin{align}\label{e5.11}
g_3(k,h)&< \frac{\sqrt{m_L}}{\sqrt{ h(m_L-h)}} \frac{k^{rh}(n-k)^{r(m_L-h)}}
{n^{rm_L}} \frac{m_L^{m_L}}{h^h(m_L-h)^{m_L-h}}.
\end{align}
Note that $k^{rh}(n-k)^{r(m_L-h)}$ attains  its maximum when $k= \frac{h n}{m_L}$.
Substitute $k= \frac{h n}{m_L}$ into the equation in~\eqref{e5.11}, thus
\begin{align*}
g_3(k,h)&< \frac{\sqrt{m_L}}{\sqrt{ h(m_L-h)}}\Bigl( \frac{h}{m_L}\Bigr)^{rh}
\Bigl(1- \frac{h}{m_L}\Bigr)^{r(m_L-h)} \frac{m_L^{m_L}}{h^h(m_L-h)^{m_L-h}}.
\end{align*}

Let
\begin{align}\label{e5.12}
g_4(h)=  \frac{\sqrt{m_L}}{\sqrt{ h(m_L-h)}}\Bigl( \frac{h}{m_L}\Bigr)^{rh}
\Bigl(1- \frac{h}{m_L}\Bigr)^{r(m_L-h)} \frac{m_L^{m_L}}{h^h(m_L-h)^{m_L-h}},
\end{align}
where $h\in \bigl[ \frac{m_L}{\log n}, \frac{m_L}{2}\bigr]$ because $k= \frac{h n}{m_L}$
and $k\in I_2$ in (4.7). Take the logarithm to $g_4(h)$ and differentiate
$\log(g_4(h))$ with respect to $h$. It follows that
\begin{align*}
g_4'(h)&=g_4(h)\Bigl( -\frac{1}{2h}+ \frac{1}{2(m_L-h)} +(r-1)
\log \frac{h}{m_L-h}\Bigr)\leqslant 0
\end{align*}
because $h\in \bigl[ \frac{m_L}{\log n}, \frac{m_L}{2}\bigr]$ and $g_4(h)\geqslant 0$. Hence, we have
 \begin{align}\label{e55.7}
 g_4(h)\leqslant g_4\Bigl( \frac{m_L}{\log n}\Bigr).
 \end{align} 
Putting $h= \frac{m_L}{\log n}$ into the equation in~\eqref{e5.12}, 
we have
\begin{align}\label{e55.8}
 \frac{\sqrt{m_L}}{\sqrt{ h(m_L-h)}}&=
\sqrt{ \frac{\log n}{m_L}}\Bigl( 1- \frac{1}{\log n}\Bigr)^{- \frac{1}{2}},\notag\\
 \Bigl( \frac{h}{m_L}\Bigr)^{rh}&=\bigl(\log n\bigr)^{- \frac{rm_L}{\log n}},\notag\\
 \Bigl(1- \frac{h}{m_L}\Bigr)^{r(m_L-h)}&=\Bigl( 1- \frac{1}{\log n}\Bigr)^{rm_L(1- \frac{1}{\log n})},\notag\\
 \frac{m_L^{m_L}}{h^h(m_L-h)^{m_L-h}}&=\bigl(\log n\bigr)^{ \frac{m_L}{\log n}}
\Bigl( 1-\frac{1}{\log n}\Bigr)^{ -m_L(1- \frac{1}{\log n})}.
\end{align}
Using the equations in~\eqref{e55.6},~\eqref{e5.12},~\eqref{e55.7} and~\eqref{e55.8},
we have
\begin{align*}
\mathbb{E}[Y_{k,h}]&<\binom{n}{ \frac{n}{2}} \frac{1}{\sqrt{m_L}}\bigl(\log n\bigr)^{ \frac{(1-r)m_L}{\log n}+ \frac{1}{2}}
\Bigl( 1-\frac{1}{\log n}\Bigr)^{ (r-1)m_L(1-\frac{1}{\log n})- \frac{1}{2}}.
\end{align*}
Since $ \binom{n}{ \frac{n}{2}}= \frac{\sqrt{2\pi n}(\frac{n}{e})^n}{\pi n(\frac{n}{2e})^n}
=O( \frac{2^n}{\sqrt{n}})$  by Stirling's formula when $n\rightarrow\infty$, $m_L= \frac{n}{r}(\log n-\omega(n))$ and
\begin{align*}
\Bigl( 1-\frac{1}{\log n}\Bigr)^{ (r-1)m_L(1-\frac{1}{\log n})- \frac{1}{2}}=O(1)
\end{align*}
when $n\rightarrow\infty$, we further have $\mathbb{E}[Y_{k,h}]
=O\bigl(n^{- 1}2^n\bigl(\log n\bigr)^{ \frac{(1-r)m_L}{\log n}}\bigr)$
to complete the proof of  Claim 7.
 \end{proof}

\bigskip
\noindent{\bf Claim 8}.~~$\sum_{k\in I_2}\mathbb{E}[Y_k]\rightarrow0$.

\begin{proof}[Proof of Claim 8]\
Fix any $k\in I_2$, we have
\begin{align*}
\mathbb{E}[Y_k]&=\sum_{h\leqslant \frac{m_L}{2}}\mathbb{E}[Y_{k,h}] + \sum_{h> \frac{m_L}{2}}\mathbb{E}[Y_{k,h}].
\end{align*}
On one hand, by the equation in (4.5), for any $h>\frac{m_L}{2}$, we have
\begin{align*}
\sum_{h> \frac{m_L}{2}}\mathbb{E}[Y_{k,h}]
&\leqslant \frac{ \binom{n}{k}}{\bigl|\mathcal{L}_r^\ell(n, m_L)\bigr|}\sum_{h> \frac{m_L}{2}}\bigl|\mathcal{L}_r^\ell(k, h)\bigr|
\cdot\bigl|\mathcal{L}_r^\ell(n-k, m_L-h)\bigr|.
\end{align*}
By the equation in (4.13) of Claim 5, we have
\begin{align*}
\sum_{h> \frac{m_L}{2}}\mathbb{E}[Y_{k,h}]
=O\bigl(\mathbb{E}[Y_{k, \frac{m_L}{2}}]\bigr)
\end{align*}
because the sum of this expression over $h> \frac{m_L}{2}$ is
bounded by a decreasing geometric series dominated by the term $h= \frac{m_L}{2}$.
On the other hand, by Claim 7, we also have
\begin{align*}
\sum_{h\leqslant \frac{m_L}{2}}\mathbb{E}[Y_{k,h}]
=O\Bigl(n^{- 1}2^nm_L(\log n)^{ \frac{(1-r)m_L}{\log n}}\Bigr).
\end{align*}
Thus,
\begin{align*}
\mathbb{E}[Y_k]&=\sum_{h\leqslant \frac{m_L}{2}}\mathbb{E}[Y_{k,h}] + \sum_{h> \frac{m_L}{2}}\mathbb{E}[Y_{k,h}]\\
&=O\Bigl(n^{- 1}2^nm_L\bigl(\log n\bigr)^{ \frac{(1-r)m_L}{\log n}}\Bigr)\\
&=O\Bigl( 2^n (\log n)^{ \frac{(1-r)m_L}{\log n}+ 1}\Bigr),
\end{align*}
where the last equality is true by  $m_L= \frac{n}{r}(\log n-\omega(n))$.
At last,
\begin{align*}
\sum_{k\in I_2}\mathbb{E}[Y_k]=
O\Bigl( n2^n (\log n)^{ \frac{(1-r)m_L}{\log n}+ 1}\Bigr)\rightarrow0
\end{align*}
because $(\log n+n\log 2+ \log\log n)\log n/(r-1)m_L\log\log n\rightarrow 0$
when $m_L= \frac{n}{r}(\log n-\omega(n))$ and $n\rightarrow\infty$. 
\end{proof}

By Claim 4, Claim 8 and Markov's inequality, we have
\begin{align*}
{\mathbb P}\Bigl[\sum_{r\leqslant k\leqslant \frac{n}{2}}Y_k>0\Bigr]\leqslant
\mathbb{E}\Bigl[\sum_{r\leqslant k\leqslant \frac{n}{2}}Y_k\Bigr]\rightarrow 0,
\end{align*}
which implies \textit{w.h.p.} all non-isolated vertices in $H$ belong to a giant component.
Combining with Lemma 4.2, we complete the proof of Lemma 4.3.
\end{proof}

To complete the proof of Theorem~\ref{t1.6} we now
let $m=m_R$ and prove that $\mathcal{L}_r^\ell(n,m_R)$ is \textit{w.h.p.} connected.




\begin{proof}[Proof of Theorem~\ref{t1.6}]\ Let $H$ be chosen uniformly at random
from $\mathcal{L}_r^\ell(n, m_{L})$. By Lemma~4.3, assume that $H$
consists of a connected component and at most $2\log n$ isolated vertices.
Let $V_1$ denote the collection of these isolated vertices in $H$.
We add $m_{R}-m_{L}$ random edges to $H$, which are denoted
  by $e_1,\cdots, e_{m_{R}-m_{L}}$ in sequence. If $\tau_o<\tau_c$ then at least one edge
 $e_j$ for $1\leqslant j\leqslant m_R-m_L$ must be added which contains only isolated vertices.

Let $H_j$ be chosen uniformly at random from $\mathcal{L}_r^\ell(n, m_L+j)$ for $1\leqslant j\leqslant m_R-m_L$.
By Theorem~\ref{t1.4}, we have the probability that $H_j$ contains $e_j$, 
\begin{align*}
\mathbb{P}\bigl[e_j\in H_j\bigr]&\leqslant  \frac{m_R}{N_0}\exp\biggl[
O\Bigl( \frac{1}{n^\ell}+ \frac{m_R^2}{n^{\ell+1}}\Bigr)\biggr].
\end{align*}
The number of choices for $e_j$ such that $e_j\subseteq V_1$
is at most $\binom{2\log n}{r}$ 
because there are at most
$2\log n$ isolated vertices in $H$. Using a union bound, if $H_{m_R-m_L}$ is chosen uniformly at
random from $\mathcal{L}_r^\ell(n, m_R)$, then we have
\begin{align*}
\mathbb{P}[\tau_{o}<\tau_{c}]&\leqslant o(1)+\bigl(m_R-m_L\bigr)\binom{2\log n}{r} \frac{m_R}{N_0}\exp\biggl[
O\Bigl( \frac{1}{n^\ell}+ \frac{m_R^2}{n^{\ell+1}}\Bigr)\biggr]\\
&=o(1)+O\Bigl( \frac{n^2(\log n)^{r+1}\log\log n}{ N_0}\Bigr)\\
&=o(1),
\end{align*}
where the first $o(1)$ comes from the failure probability of Lemma~4.3,
and the last equality is true because $r\geqslant 3$. Thus, \textit{w.h.p.}
$\tau_{o}=\tau_{c}$. Combining it with Lemma~4.2 when $m=m_R$, we have \textit{w.h.p.}
$\mathcal{S}(n,r,\ell; m_R)$ is connected to complete the proof of Theorem~\ref{t1.6}.
\end{proof}

We also have a corollary about
 the distribution on the number of isolated vertices in $\mathbb{L}_r^\ell(n, m)$ when
 $m= \frac{n}{r}\bigl(\log n+c_n)$, where $c_n\rightarrow c\in \mathbb{R}$ as $n\rightarrow\infty$.

\begin{lemma}[\cite{jason00}, Corollary~6.8]\label{l2.7}
Let $X=\sum_{\alpha\in A}I_\alpha$ be a counting variable, where $I_\alpha$ is an indicator variable.
If $\lambda\geqslant 0$ and $\mathbb{E}[X]_k\rightarrow\lambda^k$
for every $k\geqslant 1$ when $n\rightarrow\infty$, then $X\xrightarrow{d}Po(\lambda)$.
\end{lemma}

\begin{proof}[Proof of Corollary~\ref{c1.8}]\  Let $H$ be chosen uniformly at random from
$\mathcal{L}_r^\ell(n, m)$,
where  $m= \frac{n}{r}\bigl(\log n+c_n)$ and $c_n\rightarrow c\in \mathbb{R}$. Consider the factorial
moments of $X$, where $X$ denotes the number of isolated vertices in $H$. 
By Lemma~4.1 and $m= \frac{n}{r}\bigl(\log n+c_n)$, for some positive integer $t\geqslant 1$, we have
\begin{align*}
\mathbb{E}[X]_t&=[n]_t\mathbb{P}[\deg (v_1)=\cdots=\deg (v_t)=0]\notag\\
&=[n]_t\exp\biggl[- \frac{trm}{n}
+O\Bigl( \frac{m}{n^2}+ \frac{m^2}{n^{\ell+1}}\Bigr)\biggr]\\
&= \frac{[n]_t}{n^t}\exp\Bigl[- tc_n
+o(1)\Bigr],
\end{align*}
and $\mathbb{E}[X]_t\rightarrow \exp[-tc]$ when $n\rightarrow\infty$.
By Lemma~4.4,
we have $X$ tends in distribution to $Po(\lambda)$ with $\lambda=\exp[-c]$.
\end{proof}

\section{Conclusions}

For any fixed integers $r$ and $\ell$
with $3\leqslant\ell\leqslant r-1$, we have the asymptotic enumeration formula
of $\mathcal{L}_r^\ell(n, m)$ when $m=o( n^{ \frac{\ell+1}{2}})$
by similar proof with the case  $\ell=2$ in~\cite{mckay18}.
Applying the enumeration formula, we show the process
  $\mathbb{L}_r^\ell(n, m)$ has the same threshold
 of connectivity with $\mathbb{H}_r(n,m)$,
and it also becomes
connected exactly at the moment when the last
isolated vertex disappears.
What about other extremal properties of the partial Steiner $(n,r,\ell)$-systems process?
Recently, Balogh and Li~\cite{balgoh17} obtained an upper bound on the
total number of linear $r$-graphs with given girth for fixed $r\geqslant 3$.
For any fixed integer $g\geqslant 4$, Bohman and Warnke~\cite{bohman19} applied
a natural constrained random process
to typically produce a partial Steiner $(n,3,2)$-system with $(1/6-o(1))n^2$ edges and girth
larger than $g$. 
The process iteratively adds random $3$-set subject to the constraint that the girth
remains larger than $g$.
In future work, we will consider the final size of the partial Steiner $(n,r,\ell)$-system process
with some constraints on the girth.

\section*{Acknowledgement}

Most  of this work was
finished  when Fang Tian was a
visiting research fellow in Research School of Computer Science at
Australian National University. She is very grateful
for what she learned there.

\section*{Appendix:\ Proof of  Theorem~\ref{t1.4}}\label{s:4}

In the following, we are ready to prove Theorem~\ref{t1.4},
generalizing Theorem~1.4 in~\cite{mckay18}
from $\ell=2$ to $3\leqslant \ell\leqslant r-1$.
We assume $r$ and $\ell$ are any given integers such that
$3\leqslant \ell\leqslant r-1$, $m=o(n^{ \frac{\ell+1}{2}})$
and $k=o( \frac{n^{\ell+1}}{m^2})$ below.
We will also assume that $k\leqslant m$, otherwise Theorem~\ref{t1.4} is trivially true.
Let $K=K(n)$ be a fixed partial Steiner $(n,r,\ell)$-system in
$\mathcal{L}_r^\ell(n, k)$ on $[n]$  with edges $\{e_1,\cdots,e_k\}$.
Consider $H\in \mathcal{L}_r^\ell(n, m)$ chosen uniformly at random.
Let $\mathbb{P}[K\subseteq H]$
be the probability that
$H$ contains $K$ as a subhypergraph. 
Then we have
\begin{align*}\label{e8.1}
\mathbb{P}[K\subseteq H]&=\mathbb{P}[e_1,\cdots,e_k\subseteq H]\notag
\\&=\prod_{i=1}^{k} \frac{\mathbb{P}[e_1,\ldots,e_i\in H]}
{\mathbb{P}[e_1,\ldots,e_i\in H]+\mathbb{P}[e_1,\ldots,e_{i-1}\in H,e_i\notin H]}\notag
\\&=\prod_{i=1}^{k}\, \biggl(1+ \frac{\mathbb{P}[e_1,\ldots,e_{i-1}\in H,e_i\notin H]}
{\mathbb{P}[e_1,\ldots,e_i\in H]}\biggr)^{\!\!-1}.
\end{align*}
For $i=1,\ldots,k$, let $\mathcal{L}(m, \overline{e}_i)$ 
be the set of all partial Steiner $(n,r,\ell)$-systems in $\mathcal{L}_r^\ell(n, m)$
(denoted by $\mathcal{L}(m)$ below in this section) which contain edges
$e_1,\ldots,e_{i-1}$ but not edge $e_i$. Let $\mathcal{L}(m, e_i)=
\mathcal{L}(m)- \mathcal{L}(m, \overline{e}_i)$. 
We have the ratio
\begin{equation*}\label{e8.2}
 \frac{\mathbb{P}[e_1,\cdots,e_{i-1}\in H,e_i\notin H]}{\mathbb{P}[e_1,\cdots,e_i\in H]}
= \frac{|\mathcal{L}(m, \overline{e}_i)|}{|\mathcal{L}(m, e_i)|}.
\end{equation*}
Note that  $|\mathcal{L}(m)|\neq 0$ when $m=o(n^{ \frac{\ell+1}{2}})$
by Theorem~\ref{t1.3}. We show below by switching method again that
none of the denominators in the above equation are zero.

Let $H\in \mathcal{L}(m, e_i)$ with $1\leqslant i\leqslant k$. An {\it $e_i$-displacement}
is defined as 
removing the edge $e_i$ from $H$, taking any $r$-set distinct with $e_i$ of which no $\ell$ vertices
belong to the same edge and adding this $r$-set as an edge. The new graph is denoted by $H'$.
An {\it $e_i$-replacement} is the reverse of
$e_i$-displacement. An $e_i$-replacement from $H'\in \mathcal{L}(m, \overline{e}_i)$
consists of removing any edge in $E(H')-\{e_1,\cdots,e_{i-1}\}$,
then inserting $e_{i}$. We say that the
 $e_i$-replacement is {\it legal} if
$H\in \mathcal{L}(m, {e}_i)$, otherwise it is illegal.
We need a better estimation than $ N(1+O( \frac{m}{n^\ell}))$
in the proof of Lemma~\ref{l5.3} to analyze the above switching.

\noindent{\bf{Lemma A.1}}\ Assume that $m=o(n^{ \frac{\ell+1}{2}})$ and $1\leqslant i\leqslant k$.
Let $H\in\mathcal{L}(m-1)$ and $P_r$ be the set of $r$-sets  distinct from $e_i$
 of which no $\ell$ vertices belong to
the same edge of $H$. Then
\begin{align*}
|P_r|=\biggl[N- \binom{r}{\ell} m \binom{n-\ell}{r-\ell}\biggr]
\biggl(1+O\Bigl( \frac{1}{n^\ell}+ \frac{m^2}{n^{\ell+1}}\Bigr)\biggr)
\end{align*}

\begin{proof} We use inclusion-exclusion and note that any two edges of $H$ have at most
$\ell-1$ vertices in common. Let $\boldsymbol{i}_\ell=\{i_1,\cdots,i_{\ell}\}$ be the
$\ell$ vertices of an edge $e$ in $H$.
 Let $A_{(e;\boldsymbol{i}_\ell)}$ be the family of
$r$-sets of $[n]$ that contains the vertices $i_1,\ldots,i_{\ell}$ of the edge $e$. Then we have
$N-1-\sum_{\{e;\boldsymbol{i}_\ell\}}|A_{(e;\boldsymbol{i}_\ell)}|\leqslant |P_r|
\leqslant N-1-\sum_{\{e;\boldsymbol{i}_\ell\}}|A_{(e;\boldsymbol{i}_\ell)}|+ \sum_{\{e;\boldsymbol{i}_\ell\}\neq \{e';\boldsymbol{i}'_\ell\}}|A_{(e;\boldsymbol{i}_\ell)}\cap A_{(e';\boldsymbol{i}'_\ell)}|$.

Clearly, $|A_{(e;\boldsymbol{i}_\ell)}|= \binom{n-\ell}{r-\ell}$ for each edge $e$ and
 $\boldsymbol{i}_\ell\subset e$.   We
have $|P_r|\geqslant N-1- \binom{r}{\ell} (m-1) \binom{n-\ell}{r-\ell}$.
Now we consider the upper bound.

\noindent{\bf Case 1}.\ For the case $e=e'$, we have
\begin{align*}
\sum_{\{e;\boldsymbol{i}_\ell\}\neq \{e';\boldsymbol{i}'_\ell\}}
|A_{(e;\boldsymbol{i}_\ell)}\cap A_{(e';\boldsymbol{i}'_\ell)}|
&=\sum_{\alpha=0}^{\ell-1} \frac{1}{2}(m-1) \binom{r}{2\ell-\alpha} \binom{2\ell-\alpha}{\ell}
\binom{\ell}{\alpha} \binom{n-2\ell+\alpha}{r-2\ell+\alpha}\\
&=O\Bigl(m \binom{n-\ell-1}{r-\ell-1}\Bigr),
\end{align*}
where the last equality is true because $\alpha=\ell-1$
corresponds to the largest term.

\noindent{\bf Case 2.}\ For the case $e\neq e'$ and $|\boldsymbol{i}_\ell\cap \boldsymbol{i}'_\ell|=s$
with some $0\leqslant s\leqslant \ell-1$, we have
\begin{align*}
\sum_{s=0}^{\ell-1}\sum_{\{e;\boldsymbol{i}_\ell\}\neq \{e';\boldsymbol{i}'_\ell\}}
|A_{(e;\boldsymbol{i}_\ell)} \cap A_{(e';\boldsymbol{i}'_\ell)}|
&=\sum_{s=0}^{\ell-1}\sum_{\{x_1,\cdots,x_s\}\in \binom{[n]}{s}}
\binom{\codeg (x_1,\cdots,x_s)}{2}
\binom{r-s}{\ell-s}^2 \binom{n-2\ell+s}{r-2\ell+s}\\
&=O\Bigl(m^2\binom{n-\ell-1}{r-\ell-1}\Bigr),
\end{align*}
where the last equality is true because $\sum_{\{x_1,\cdots,x_s\}\in \binom{[n]}{s}}
\binom{\codeg (x_1,\cdots,x_s)}{2}= O(m^2)$
and $s=\ell-1$ corresponds to the largest term.
At last, we have
\begin{align*}
\sum_{\{e;\boldsymbol{i}_\ell\}\neq \{e';\boldsymbol{i}'_\ell\}}
|A_{(e;\boldsymbol{i}_\ell)}\cap A_{(e';\boldsymbol{i}'_\ell)}|=O\Bigl(m^2 \binom{n-\ell-1}{r-\ell-1}\Bigr).
\end{align*}
We complete the proof of Lemma A.1 by $  \binom{n-\ell-1}
{r-\ell-1}/\binom{n}{r}= O( \frac{1}{n^{\ell+1}})$ and
$  \binom{n-\ell}{r-\ell}/ \binom{n}{r}=O( \frac{1}{n^{\ell}})$.
\end{proof}

\noindent{\bf {Lemma A.2}}\ Assume that $m=o(n^{ \frac{\ell+1}{2}})$
and $1\leqslant i\leqslant k$. Consider $H'\in \mathcal{L}(m, \overline{e}_i)$
chosen uniformly at random. Let $E^*$
be the set of  $r$-sets $e^*$ of $[n]$ such that $|e^*\cap e_i|\geqslant \ell$.
Then, as $n\rightarrow\infty$,
\begin{align*}
\mathbb{P}\bigl[E^*_i\cap H'\neq\emptyset\bigr]= \frac{(m-i+1)
\binom{r}{\ell} \binom{n-r}{r-\ell}}{N}+O\Bigl( \frac{m^2}{n^{\ell+1}}\Bigr).
\end{align*}

\begin{proof} Fix an $r$-set $e^*\in E^*$.
  Let $\mathcal{L}(m, \overline{e}_i,e^*)$
be the set of all $r$-graphs in $\mathcal{L}(m, \overline{e}_i)$
which contain the edge $e^*$. Let $\mathcal{L}(m, \overline{e}_i,\overline{e}^*)=
\mathcal{L}(m, \overline{e}_i)-\mathcal{L}(m, \overline{e}_i,e^*)$.
Thus, we have
\begin{equation*}\label{e8.3}
\begin{aligned}[b]
\mathbb{P}[e^*\in H']
&= \frac{|\mathcal{L}(m, \overline{e}_i,e^*)|}
{|\mathcal{L}(m, \overline{e}_i,e^*)|+
|\mathcal{L}(m, \overline{e}_i,\overline{e}^*)|}
= \biggl(1+ \frac{|\mathcal{L}(m, \overline{e}_i,\overline{e}^*)|}
{|\mathcal{L}(m, \overline{e}_i,e^*)|}\biggr)^{\!-1}.
\end{aligned}
\end{equation*}

Let $G\in \mathcal{L}(m, \overline{e}_i,e^*)$ and $\mathcal{R}(G)$ be the set
of all ways to move the edge $e^*$ to an $r$-set of $[n]$ distinct from $e^*$ and $e_i$,
of which no $\ell$ vertices are in any remaining  edges of $G$. Call the new graph as $G'$.
By the same proof as Lemma A.1, we have
\begin{align*}
\mathcal{R}(G)= \biggl[N- \binom{r}{\ell} m \binom{n-\ell}{r-\ell}\biggr]
\Bigl(1+O\Bigl( \frac{1}{n^\ell}+ \frac{m^2}{n^{\ell+1}}\Bigr)\Bigr).
\end{align*}

Conversely, let $G'\in\mathcal{L}(m, \overline{e}_i, \overline{e}^*)$
and let $\mathcal{R}'(G')$ be the set of all ways to move one edge in $E(G')-\{e_1,\ldots,e_{i-1}\}$
to $e^*$ such that the resulting graph is in
$\mathcal{L}(m, \overline{e}_i, e^*)$. We apply the same switching to analyze
 the expected number of $\mathcal{R}'(G')$.
Likewise, let $E^{**}$ be the set of $r$-sets $e^{**}$ of $[n]$ such that $|e^{**}\cap e^*|\geqslant \ell$
and fix an $r$-set $e^{**}\in E^{**}$.
 Let $\mathcal{L}(m,\overline{e}_i, \overline{e}^*, {e}^{**})=
 \{H\in \mathcal{L}(m, \overline{e}_i,\overline{e}^*):\ e^{**}\in H\}$
and $\mathcal{L}(m,\overline{e}_i, \overline{e}^*,\overline{e}^{**})=
\mathcal{L}(m,\overline{e}_i,\overline{e}^*)-\mathcal{L}(m, \overline{e}_i, \overline{e}^*,{e}^{**})$.
We also have $\mathbb{P}[e^{**}\in G']
= \bigl(1+ \frac{|\mathcal{L}(m,\overline{e}_i, \overline{e}^*,\overline{e}^{**})|}
{|\mathcal{L}(m, \overline{e}_i, \overline{e}^*,{e}^{**})|}\bigr)^{\!-1}$.
For a hypergraph in $\mathcal{L}(m, \overline{e}_i, \overline{e}^*,{e}^{**})$,
by the same proof as Lemma A.1, we also have $[N- \binom{r}{\ell}m
\binom{n-\ell}{r-\ell}](1+O( \frac{1}{n^\ell}+ \frac{m^2}{n^{\ell+1}}))$
ways to move the edge $e^{**}$ to an $r$-set of $[n]$ distinct from ${e}_i$, ${e}^*$ and ${e}^{**}$
 such tha no $\ell$ vertices are in any remaining edges. Similarly, there are at most $m-i+1$ ways
to switch a hypergraph from $\mathcal{L}(m,\overline{e}_i, \overline{e}^*,\overline{e}^{**})$ to
$\mathcal{L}(m, \overline{e}_i, \overline{e}^*,{e}^{**})$.
We have
$\mathbb{P}[e^{**}\in G']
= O\bigl( \frac{m}{N}\bigr)$.  Note that $|E^{**}|=
\sum_{i=\ell}^{r-1}\binom{r}{i} \binom{n-r}{r-i}=O\bigl(  \binom{n-r}{r-\ell}\bigr)$,
then $\mathbb{P}[E^{**}\cap G'\neq\emptyset]=O\bigl( \frac{m}{n^\ell}\bigr)$ and
the expected number of $\mathcal{R}'(G')$ is $(m-i+1)\bigl(1-O\bigl( \frac{m}{n^\ell}\bigr)\bigr)$.
Thus, we have
\begin{align*}
\frac{|\mathcal{L}(m, \overline{e}_i, \overline{e}^*)|}{|\mathcal{L}(m, \overline{e}_i, e^*)|}
&= \frac{|\mathcal{R}(G)|}{|\mathcal{R}'(G')|}= \frac{N- \binom{r}{\ell}m\binom{n-\ell}{r-\ell}}{m-i+1}
\biggl(1+O\Bigl( \frac{m}{n^\ell}+ \frac{m^2}{n^{\ell+1}}\Bigr)\biggr)\\
&= \frac{N}{m-i+1}\biggl(1+O\Bigl( \frac{m}{n^\ell}+ \frac{m^2}{n^{\ell+1}}\Bigr)\biggr),
\end{align*}
and we also have
\begin{align*}
\mathbb{P}[e^*\in H']=\Bigl(1+ \frac{|\mathcal{L}(m, \overline{e}_i, \overline{e}^*)|}
{|\mathcal{L}(m, \overline{e}_i, e^*)|}\Bigr)^{\!-1}= \frac{m-i+1}{N}
\Bigl(1+O\Bigl( \frac{m}{n^\ell}+ \frac{m^2}{n^{\ell+1}}\Bigr)\Bigr).
\end{align*}

By inclusion-exclusion,
\begin{equation*}
\begin{aligned}
\sum_{e^*\in E^*}&\mathbb{P}[e^*\in H']-\sum_{e_1^*\in E^*, e_2^*\in E^*,e_1^*\cap e_2^*\neq\emptyset}
\mathbb{P}[e_1^*,e_2^*\in H']\notag\\
 &{}\leqslant \mathbb{P}\bigl[E^*\cap H'\neq\emptyset\bigr]\leqslant \sum_{e^*\in E^*}\mathbb{P}[e^*\in H'].
\end{aligned}
\eqno(A.1)
\end{equation*}
Since $|E^*|=\sum_{i=\ell}^{r-1} \binom{r}{i} \binom{n-r}{r-i}= \binom{r}{\ell}
 \binom{n-r}{r-\ell}+O( n^{r-\ell-1})$, we have
 \begin{align*}
 \sum_{e^*\in E^*}\mathbb{P}[e^*\in H']=
 \frac{(m-i+1) \binom{r}{\ell} \binom{n-r}{r-\ell}}{N}
 +O\Bigl( \frac{m^2}{n^{\ell+1}}\Bigr)
 \end{align*}
because $O\bigl( \frac{m}{N} \binom{r}{\ell} \binom{n-r}{r-\ell}
( \frac{m}{n^\ell}+ \frac{m^2}{n^{\ell+1}})\bigr)=
O\bigl( \frac{m^2}{n^{\ell+1}}\bigr)$ and $O\bigl( \frac{m}{N}
n^{r-\ell-1}\bigr)=O\bigl( \frac{m^2}{n^{\ell+1}}\bigr)$.

Consider
$\sum_{e_1^*\in E^*, e_2^*\in E^*, e_1^*\cap e_2^*\neq\emptyset}\mathbb{P}[e_1^*,e_2^*\in H']$.
Note that $H'\in \mathcal{L}(n,r,\ell; m, \overline{e}_i)$, then $|e_1^*\cap e_2^*|\leqslant \ell-1$.
Let $\mathcal{L}(m, \overline{e}_i, e_1^*, e_2^*)$,
$\mathcal{L}(m,\overline{e}_i, {e}_1^*, \overline{e}_2^*)$,
$\mathcal{L}(m, \overline{e}_i, \overline{e}_1^*, e_2^*)$ and
$\mathcal{L}(m, \overline{e}_i, \overline{e}_1^*, \overline{e}_2^*)$
be the set of all partial Steiner $(n,r,\ell)$-systems in $\mathcal{L}(m, \overline{e}_i)$ which contain both
$e_1^*$ and  $e_2^*$, only contain $e_1^*$, only contain $e_2^*$ and neither of them,
respectively. Thus, we have
\begin{align*}
&\mathbb{P}[e_1^*,e_2^*\in H']
= \frac{|\mathcal{L}(m, \overline{e}_i, e_1^*, e_2^*)|}{|\mathcal{L}(m, \overline{e}_i)|}\notag\\
&= \frac{|\mathcal{L}(m, \overline{e}_i, e_1^*, e_2^*)|}
{|\mathcal{L}(m, \overline{e}_i, e_1^*, e_2^*)|+|\mathcal{L}
(m, \overline{e}_i, {e}_1^*, \overline{e}_2^*)|+|\mathcal{L}(m, \overline{e}_i, \overline{e}_1^*, e_2^*)|
+|\mathcal{L}(m, \overline{e}_i, \overline{e}_1^*, \overline{e}_2^*)|}\notag\\
&= \biggl(1+ \frac{|\mathcal{L}(m, \overline{e}_i, {e}_1^*, \overline{e}_2^*)|}
{|\mathcal{L}(m, \overline{e}_i, e_1^*, e_2^*)|}+ \frac{|\mathcal{L}(m, \overline{e}_i,
\overline{e}_1^*, e_2^*)|}{|\mathcal{L}(m, \overline{e}_i, e_1^*, e_2^*)|}+
\frac{|\mathcal{L}(m, \overline{e}_i, \overline{e}_1^*, \overline{e}_2^*)|}
{|\mathcal{L}(m, \overline{e}_i, e_1^*, e_2^*)|}\biggr)^{\!-1}. 
\end{align*}
By the similar analysis above, we have
\begin{align*}
 \frac{|\mathcal{L}(m,\overline{e}_i, {e}_1^*, \overline{e}_2^*)|}
 {|\mathcal{L}(m, \overline{e}_i, e_1^*, e_2^*)|}&\geqslant
 \frac{[N- \binom{r}{\ell}m \binom{n-\ell}{r-\ell}]}
 {m-i+1}\Bigl(1+O\Bigl( \frac{1}{n^\ell}+ \frac{m^2}{n^{\ell+1}}\Bigr)\Bigr),\\
 \frac{|\mathcal{L}(m,\overline{e}_i, \overline{e}_1^*, e_2^*)|}
 {|\mathcal{L}(m,\overline{e}_i, e_1^*, e_2^*)|}&\geqslant
 \frac{[N- \binom{r}{\ell}m \binom{n-\ell}{r-\ell}]}{m-i+1}
 \Bigl(1+O\Bigl( \frac{1}{n^\ell}+ \frac{m^2}{n^{\ell+1}}\Bigr)\Bigr).
 \end{align*}
For any hypergraph in $\mathcal{L}(m, \overline{e}_i, e_1^*, e_2^*)$, we move $e_1^*$ and $e_2^*$ away
by the $e_1^*$-displacement and $e_2^*$-displacement operations. For $e_1^*$ (resp. $e_2^*$),
by the same proof as Lemma A.1,
there are  $[N- \binom{r}{\ell}m \binom{n-\ell}{r-\ell}]
(1+O( \frac{1}{n^\ell}+ \frac{m^2}{n^{\ell+1}}))$ ways to move  $e_1^{*}$
 (resp. $e_2^*$)  to an $r$-set of $[n]$ distinct from ${e}_i$, $e_1^*$ and $e_2^*$
 such that the resulting graph is in $\mathcal{L}(m, \overline{e}_i, \overline{e}_1^*, \overline{e}_2^*)$.
Similarly, there are at most $2 \binom{m-i+1}{2}$ ways
to switch a hypergraph from $\mathcal{L}(m, \overline{e}_i, \overline{e}_1^*, \overline{e}_2^*)$ to
$\mathcal{L}(m, \overline{e}_i, e_1^*, e_2^*)$.
Thus, we have
\begin{align*}
\frac{|\mathcal{L}(m, \overline{e}_i, \overline{e}_1^*, \overline{e}_2^*)|}
 {|\mathcal{L}(m, \overline{e}_i, e_1^*, e_2^*)|}
 \geqslant \frac{[N- \binom{r}{\ell}m \binom{n-\ell}{r-\ell}]^2}
 {2\binom{m-i+1}{2}}\Bigl(1+O\Bigl( \frac{1}{n^\ell}+ \frac{m^2}{n^{\ell+1}}\Bigr)\Bigr).
 \end{align*}
Note that there are
 $O( n^{2r-2\ell})$ ways to choose the pair $\{e_1^*,e_2^*\}$
 such that $|e_1^*\cap e_i|\geqslant \ell$, $|e_2^*\cap e_i|\geqslant \ell$ and $|e_1^*\cap e_2^*|\leqslant\ell-1$,
then we have
\begin{align*}
\sum_{e_1^*\in E^*, e_2^*\in E^*,e_1^*\cap e_2^*\neq\emptyset}
\mathbb{P}[e_1^*,e_2^*\in H']=O\Bigl( n^{2r-2\ell} \frac{m^2}{N^2}\Bigr)=O\Bigl( \frac{m^2}{n^{\ell+1}}\Bigr).
\end{align*}
To complete the proof of Lemma A.2, add together the above
equations into inclusion-exclusion formula (A.1).
\end{proof}

By Lemma A.1 and Lemma A.2, we finally have

\noindent{\bf Corollary A.3}\ For any given integers $r$ and $\ell$ such that
$3\leqslant\ell\leqslant r-1$, assume that
 $m=o(n^{ \frac{\ell+1}{2}})$ and $1\leqslant i\leqslant k$.
With notation above, as $n\rightarrow\infty$,\\
\noindent$(a)$\ Let $H\in \mathcal{L}(m, {e}_i)$. The number of $e_i$-displacements  is
$[N- \binom{r}{\ell} m \binom{n-\ell}{r-\ell}]
(1+O( \frac{1}{n^\ell}+ \frac{m^2}{n^{\ell+1}}))$.
\noindent$(b)$\ Consider $H'\in\mathcal{L}(m, \overline{e}_i)$ chosen uniformly at random.
The expected number of legal $e_i$-replacements is
$(m-i+1)\bigl(1- \frac{(m-i+1)\binom{r}{\ell} \binom{n-r}{r-\ell}}{N}+O( \frac{m^2}{n^{\ell+1}})\bigr)$.

\noindent$(c)$\  $\frac{|\mathcal{L}(m,\overline{e}_i)|}{|\mathcal{L}(m,{e}_i)|}
 = \frac{[N- \binom{r}{\ell} m\binom{n-\ell}{r-\ell}]}{m-i+1}\bigl(1+ \frac{(m-i+1) \binom{r}{\ell} \binom{n-r}{r-\ell}}{N}
+O\bigl( \frac{1}{n^\ell}+ \frac{m^2}{n^{\ell+1}}\bigr)\bigr)$.

By Corollary A.3 $(c)$, as the equations shown below, we have
\begin{align*}
&\mathbb{P}[K\subseteq H]\\
&=\prod_{i=1}^{k} \frac{m-i+1}{\bigl[N- \binom{r}{\ell} m
\binom{n-\ell}{r-\ell}\bigr]}\biggl[1- \frac{(m-i+1)\binom{r}{\ell}
\binom{n-r}{r-\ell}}{N}
+O\Bigl( \frac{1}{n^\ell}+ \frac{m^2}{n^{\ell+1}}\Bigr)\biggr]\\
&=\prod_{i=1}^{k} \frac{m-i+1}{\bigl[N- \binom{r}{\ell}m \binom{n-\ell}{r-\ell}\bigr]}
\exp\biggl[- \frac{(m-i+1)\binom{r}{\ell} \binom{n-r}{r-\ell}}{N}
+O\Bigl( \frac{1}{n^\ell}+ \frac{m^2}{n^{\ell+1}}\Bigr)\biggr]\\
&= \frac{[m]_k}{N^k}\exp\biggl[ \frac{[r]_{\ell}^2k^2}{2\ell!n^{\ell}}
+O\Bigl( \frac{k}{n^\ell}+ \frac{m^2k}{n^{\ell+1}}\Bigr)\biggr],
\end{align*}
where $k=o\bigl(\frac{n^{\ell+1}}{m^2}\bigr)$. We complete the proof of Theorem~\ref{t1.4}.

\end{document}